\documentclass[a4paper,12pt,reqno]{amsart}
\usepackage{amsfonts}
\usepackage{amsmath}
\usepackage{amssymb}
\usepackage[a4paper]{geometry}
\newtheorem{theorem}{Theorem}
\newtheorem{definition}{Definition}

\newtheorem{remark}{Remark}

\usepackage{comment}
\usepackage{mathrsfs}
\usepackage{csquotes}
\usepackage{enumitem}
\usepackage{dirtytalk}

\usepackage[colorlinks]{hyperref}
\renewcommand\eqref[1]{(\ref{#1})} 
\numberwithin{equation}{section}
\theoremstyle{plain}

\theoremstyle{definition}

\begin{document}
  \title[Cylindrical Hardy, Sobolev type and Caffarelli-Kohn-Nirenberg]
 {Cylindrical Hardy, Sobolev type and Caffarelli-Kohn-Nirenberg type inequalities and identities}

\author[M. Kalaman]{Madina Kalaman}
\address{
  Madina Kalaman:
  \endgraf
  SDU University, Kaskelen, Kazakhstan
  \endgraf
  and 
  \endgraf
  Institute of Mathematics and Mathematical Modeling, Almaty, Kazakhstan
    \endgraf
  {\it E-mail address} {\rm madina.kalaman22@gmail.com}
  }

\author[N. Yessirkegenov]{Nurgissa Yessirkegenov}
\address{
  Nurgissa Yessirkegenov:
  \endgraf
  SDU University, Kaskelen, Kazakhstan
  \endgraf
  and 
  \endgraf
  Institute of Mathematics and Mathematical Modeling, Almaty, Kazakhstan
    \endgraf
  {\it E-mail address} {\rm nurgissa.yessirkegenov@gmail.com}
  }

\thanks{This research is funded by the Science Committee of the Ministry of Science and Higher Education of the Republic of Kazakhstan (Grant No. AP23490970).}


     \begin{abstract} 
     In this paper we discuss cylindrical extensions of improved Hardy, Sobolev type and Caffarelli-Kohn-Nirenberg type inequalities with sharp constants and identities in the spirit of Badiale-Tarantello \cite{BT}. All identities are obtained in the setting of $L^p$ for all $p\in (1,\infty)$ without the real-valued function assumption. The obtained identities provide a simple
and direct understanding of these inequalities as well as the nonexistence
of nontrivial extremizers. As a byproduct, we show extended Caffarelli-Kohn-Nirenberg type inequalities with remainder terms that imply a cylindrical extension of the classical Heisenberg-Pauli-Weyl uncertainty principle. Furthermore, we prove $L^p$-Hardy type identities with logarithmic weights that imply the critical Hardy inequality in the special case. Lastly, we also discuss extensions of these results on homogeneous Lie groups. Particular attention is paid to stratified Lie groups, where the functional inequalities become intricately intertwined with the properties of sub-Laplacians and related hypoelliptic partial differential equations. The obtained results already imply new insights even in the classical Euclidean setting with respect to the range of parameters and the arbitrariness of the choice of any homogeneous quasi-norm. 

     \end{abstract}
     \maketitle

\section{Introduction}\label{eq:1}

One of the main aims of the paper is inspired by the work of Badiale-Tarantello \cite{BT}, namely by the following extended Hardy inequality: Let $x=\left({x^{\prime}}, x^{\prime \prime}\right) \in {\mathbb{R}^k} \times \mathbb{R}^{n-k}$, $2 \leq k \leq n$ and $1 \leq p<k$. Then, there exists a positive constant $C_{n, k, p}$ such that
\begin{equation}\label{B-T}
\left\|\frac{f}{{|x^{\prime}|_k}} \right\|_{L^p\left(\mathbb{R}^n\right)} \leq C_{n, k, p}\|\nabla f\|_{L^p\left(\mathbb{R}^n\right)},
\end{equation}
where $\nabla$ is the full gradient on ${\mathbb{R}^{n}}$ and $|x'|_k$ is the Euclidean norm on {$\mathbb{R}^{k}$}. There, in \cite{BT} the best constant $C_{k, p}=\frac{p}{k-p}$ in \eqref{B-T} was conjectured and then verified in \cite{SDM} and recently in \cite{RS17_strat} by a different method. We also refer to \cite{MS04} for extremizers of cylindrical Hardy-Sobolev inequalities as well as to \cite{Lam18}, \cite{DN20}, \cite{DP21} for cylindrical Hardy inequalities with Bessel pairs. Badiale and Tarantello in \cite{BT} showed an application of such cylindrical Hardy type inequalities in investigating existence and non-existence of cylindrical solutions
for a nonlinear elliptic equation, which has been proposed by  Bertin \cite{Ber00} and Ciotti \cite{Cio01} as a model for
the dynamics of galaxies.

Maz'ya in \cite[Sect. 2.1.7, Corollary 3]{Maz11} obtained the extended Hardy inequality \eqref{B-T} with a remainder term when $p=2$: Let $n>2,2<q<\frac{2n}{n-2}$, and $\gamma=-1+n\left(\frac{1}{2}-\frac{1}{q}\right)$. Then there exists a positive constant $c$ such that
\begin{equation}\label{Maz_ineq}
\left(\frac{k-2}{2}\right)^{2}\left\|\frac{f}{|x'|}\right\|_{L^{2}(\mathbb{R}^{n})}^{2}\leq \|\nabla f\|^{2}-c\left(\int_{\mathbb{R}^{n}}|u|^q|x'|_{k}^{\gamma q} \mathrm{~d} x\right)^{\frac{2}{q}}
\end{equation}
holds for any $f \in C_0^{\infty}\left(\mathbb{R}^{n}\right)$, subject to the condition $f(0, x'')=0$ in the case $k=1$. 

In this direction, in particular, we discuss improved versions of \eqref{B-T} with $\frac {x' \cdot \nabla_k f}{|x'|_{k}}$ instead of $\nabla f$ and with an explicit remainder term, where $\nabla_{k}=(\partial_{x_{1}},\ldots,\partial_{x_{k}})$. Actually, we obtain the following weighted $L^p$ identity that implies the improved version: Let $1<p<\infty$. Then, for all complex-valued functions $f\in C_0^\infty(\mathbb{R}^{n} \backslash \{x'=0\})$ on $\mathbb{R}^k \times \mathbb{R}^{n-k}$ we have
\begin{multline}\label{Lp_Hardy}
\left|{\frac{k-\alpha}{p}}\right|^{p}\left\|\frac{f}{|x'|_{k}^{\frac{\alpha}{p}}}\right\|_{L^{p}(\mathbb{R}^{n})}^p =\left\|\frac {x' \cdot \nabla_k}{|x'|_{k}^{\frac{{\alpha}}{p}}}f\right\|_{{L}^{p}(\mathbb{R}^{n})}^p\\-\int_{\mathbb{R}^{n}}C_{p}\left(\frac{x'\cdot\nabla_{k}}{|x'|_{k}^{\frac{\alpha}{p}}}f, \frac{x'\cdot\nabla_{k}}{|x'|_{k}^{\frac{\alpha}{p}}}f+\frac{k-\alpha}{p}\frac{f}{|x'|_{k}^{\frac{\alpha}{p}}}\right)dx,
\end{multline}
where \begin{equation}\label{Cp_}
C_{p}(\xi,\eta):=|\xi|^{p}-|\xi-\eta|^{p}-p|\xi-\eta|^{p-2}\operatorname{Re}(\xi-\eta)\cdot \overline{\eta}.
\end{equation}
It is known that $C_{p}\geq 0$ from \cite[Step 3 of Proof of Lemma 3.4]{CKLL24}. We also show that when $k\neq\alpha$ the constant $\left|{\frac{k-\alpha}{p}}\right|^{p}$ is sharp in the inequality derivied from \eqref{Lp_Hardy} after dropping the last remainder term.  

Applying the Schwarz inequality and using $|\nabla_{k}f|\leq |\nabla f|$ we can derive from \eqref{Lp_Hardy} the following (weighted) cylindrical extended Hardy inequality with an explicit remainder term:
\begin{multline}\label{Lp_Hardy_rem}
\left|{\frac{k-\alpha}{p}}\right|^{p}\left\|\frac{f}{|x'|_{k}^{\frac{\alpha}{p}}}\right\|_{L^{p}(\mathbb{R}^{n})}^p \leq \left\|\frac {\nabla f}{|x'|_{k}^{\frac{{\alpha}}{p}-1}}\right\|_{{L}^{p}(\mathbb{R}^{n})}^p\\-\int_{\mathbb{R}^{n}}C_{p}\left(\frac{x'\cdot\nabla_{k}}{|x'|_{k}^{\frac{\alpha}{p}}}f, \frac{x'\cdot\nabla_{k}}{|x'|_{k}^{\frac{\alpha}{p}}}f+\frac{k-\alpha}{p}\frac{f}{|x'|_{k}^{\frac{\alpha}{p}}}\right)dx.
\end{multline}
Note that when $\alpha=p$ dropping the last term above implies the extended Hardy inequality \eqref{B-T}. Also, the identity \eqref{Lp_Hardy} extends the results \cite[Corollary 1.1 and 1.2]{DP21} from $L^{2}$ to general $L^{p}$. There the authors used the factorization method that cannot be applied in the $L^{p}$ setting. It is worth mentioning that there are many works on $L^{p}$-Hardy identities for real-valued functions. However, as for the $L^{p}$-Hardy identities for complex-valued functions, to the best of our knowledge, only recently the paper \cite[Lemma 3.3]{CKLL24} introduced \eqref{Lp_Hardy} with the full gradient $\nabla f$ instead of $x' \cdot \nabla_k$ when $\alpha=p$ and $k=n$. 

Moreover, dropping the last term in the identity \eqref{Lp_Hardy} when $k=n$ implies the Hardy inequalities for the radial derivative operator $df/d|x|_{E}$ from e.g. \cite{MOW13} and \cite{MOW19}, where $|x|_{E}$ is the Euclidean norm on $\mathbb{R}^{n}$. Namely, when $k=n$ taking into account $$\frac{{x'} \cdot\nabla_k f}{|x'|_k}=\frac{{x} \cdot\nabla f}{|x|_E}=\frac{df}{d|x|_E},$$ we derive from \eqref{Lp_Hardy} the Hardy identity for all $f\in C_0^\infty(\mathbb{R}^{n} \backslash \{x'=0\})$ with the radial derivative operator 
\begin{multline}\label{Lp_Hardy_radial}
\left|{\frac{n-\alpha}{p}}\right|^{p}\left\|\frac{f}{|x|_{E}^{\frac{\alpha}{p}}}\right\|_{L^{p}(\mathbb{R}^{n})}^p =\left\|\frac{1}{|x|_{E}^{\frac{\alpha}{p}-1}}\frac {df}{d|x|_{E}}\right\|_{{L}^{p}(\mathbb{R}^{n})}^p\\-\int_{\mathbb{R}^{n}}C_{p}\left(\frac{1}{|x|_{E}^{\frac{\alpha}{p}-1}}\frac {df}{d|x|_{E}}, \frac{1}{|x|_{E}^{\frac{\alpha}{p}-1}}\frac {df}{d|x|_{E}}+\frac{n-\alpha}{p}\frac{f}{|x|_{E}^{\frac{\alpha}{p}}}\right)dx,
\end{multline}
which implies the following improved version of the classical Hardy inequality when $\alpha=p$:
\begin{equation}\label{L2_Hardy_improved}
\left({\frac{{n}-{p}}{p}}\right)^{p}{{\left\|\frac{f}{|x|_E}\right\|_{L^{p}(\mathbb{R}^{n})}^p \leq \left\|\frac{df}{d|x|_E}\right\|_{{L}^{p}(\mathbb{R}^{n})}^p}}\leq \|\nabla f\|^{p}_{L^{p}(\mathbb{R}^{n})}.
\end{equation}

Also, the identity \eqref{Lp_Hardy} implies the Sobolev type inequality from \cite{OS09} (see also \cite{BEHL08}) after dropping the remainder term when $\alpha=0$ and $k=n$. Therefore, the identity \eqref{Lp_Hardy} can be thought as a sharp remainder formula for the cylindrical weighted Sobolev type inequality.  

Furthermore, in this paper we will discuss the following Hardy type identities with logarithmic weights on $\{x=(x',x'')\in \mathbb{R}^{k}\times\mathbb{R}^{n-k}:|x'|<R\}$ for any $R>0$ and for all complex-valued functions $f\in C_{0}^{\infty}(\{0<|x'|_{k}<R\})$:
\begin{multline}\label{Lp_logarithmic}
\left|\frac{\beta+p}{p}\right|^{p}\quad   \left\|\frac{f}{|x'|_{k}^{\frac{\alpha+p}{p}}\left(log\frac{R}{|x'|_{k}}\right)^{\frac{\beta+p+1}{p}}}\right\|_{L^{p}(0<|x'|_{k}<R)}^p=\left\|\frac {x' \cdot \nabla_{k}}{|x'|_{k}^{\frac{\alpha+p}{p}}\left(log\frac{R}{|x'|_{k}}\right)^{\frac{\beta+1}{p}}}f\right\|_{{L}^{p}(0<|x'|_{k}<R)}^p  \\-(k-\alpha-p)\left(\frac{\beta+p}{p}\right)\left|\frac{\beta+p}{p}\right|^{p-2}\int_{0<|x'|_{k}<R}\frac{|f|^p}{|x'|_{k}^{\alpha+p}\left(log\frac{R}{|x'|_{k}}\right)^{\beta+p}}dx
\\-\int_{0<|x'|_{k}<R}C_p\left(\xi,\eta\right)dx,   \end{multline}
where $$
\xi:=\frac{x'\cdot\nabla_{k}}{|x'|_{k}^{\frac{\alpha+p}{p}}\left(log\frac{R}{|x'|_{k}}\right)^{\frac{\beta+1}{p}}}f
$$
and
$$
\eta:=\frac{x'\cdot\nabla_{k}}{|x'|_{k}^{\frac{\alpha+p}{p}}\left(log\frac{R}{|x'|_{k}}\right)^{\frac{\beta+1}{p}}}f+\frac{\beta+p}{p}\frac{f}{|x'|_{k}^{\frac{\alpha+p}{p}}\left(log\frac{R}{|x'|_{k}}\right)^{\frac{\beta+p+1}{p}}}.
$$
The special case $p=2$, $\beta=-1$ and $\alpha=k-2$ of \eqref{Lp_logarithmic} was obtained in \cite[Corollary 1.3]{DP21} by the factorization method, which does not work in general $L^{p}$. Also, note that when $k=n$ dropping the last term in \eqref{Lp_logarithmic} implies \cite[Theorem 3]{Sano_arxiv}. So, in the special case $k=n$, the identity \eqref{Lp_logarithmic} gives a sharp remainder formula for \cite[Theorem 3]{Sano_arxiv}. When $\alpha=0$ and $\beta=-1$ dropping the last two terms in \eqref{Lp_logarithmic} we derive the following inequality
$$\left\|\frac {x' \cdot \nabla_{k}}{|x'|_{k}}f\right\|_{{L}^{p}(0<|x'|_{k}<R)}^p \geq \left|\frac{p-1}{p}\right|^{p}\quad   \left\|\frac{f}{|x'|_{k}log\frac{R}{|x'|_{k}}}\right\|_{L^{p}(0<|x'|_{k}<R)}^p,$$
which goes to the critical Hardy inequality as $p\rightarrow n$ when $k=n$. According to \cite{Sano_arxiv} the last inequality has a scale invariance structure. We also refer to \cite{II14} and \cite{BT19} for discussions on the non-attainability of the best constant in the last inequality when $p=k=n$. 

By using the obtained identities we give a characterization of the class of functions
that makes the remainder term vanish, which allows us to show nonexistence
of nontrivial extremizers in the corresponding inequalities. 

As an application, we show extended Caffarelli-Kohn-Nirenberg (CKN) type inequalities with remainder terms, which immediately imply the extended CKN inequalities from \cite{RSY18}, \cite{RSY17} or \cite{RSY17_strat}. Nowadays, there are many works concerning various versions of the Caffarelli-Kohn-Nirenberg type inequalities and applications, see e.g. \cite{ChZh, DEFT, DEL, D, LL} and the recent paper \cite{CFLL22}.

To put the extended CKN type inequalities in perspective, let us first recall the classical version from \cite{CKN}:
\begin{theorem}\cite{CKN}\label{clas_CKN} Let $n\in\mathbb{N}$ and let $p$, $q$, $r$, $a$, $b$, $d$, $\delta\in \mathbb{R}$ such that $p,q\geq1$, $r>0$, $0\leq\delta\leq1$, and

\begin{equation}\label{clas_CKN0}
\frac{1}{p}+\frac{a}{n},\, \frac{1}{q}+\frac{b}{n},\, \frac{1}{r}+\frac{c}{n}>0,
\end{equation}
where $c=\delta d + (1-\delta) b$. Then there exists a positive constant $C$ such that
\begin{equation}\label{clas_CKN_ineq}
{\||x|_{E}^{c}f\|_{L^{r}(\mathbb{R}^n)}\leq C \||x|_{E}^{a}|\nabla f|\|^{\delta}_{L^{p}(\mathbb{R}^n)} \||x|_{E}^{b}f\|^{1-\delta}_{L^{q}(\mathbb{R}^n)}},
\end{equation}
holds for all $f\in C_{0}^{\infty}(\mathbb{R}^n)$, if and only if the following conditions hold:
\begin{equation}
\frac{1}{r}+\frac{c}{n}=\delta \left(\frac{1}{p}+\frac{a-1}{n}\right)+(1-\delta)\left(\frac{1}{q}+\frac{b}{n}\right),
\end{equation}
\begin{equation}
a-d\geq 0 \quad {\rm if} \quad \delta>0,
\end{equation}
\begin{equation}
a-d\leq 1 \quad {\rm if} \quad \delta>0 \quad {\rm and} \quad \frac{1}{r}+\frac{c}{n}=\frac{1}{p}+\frac{a-1}{n}.
\end{equation}
\end{theorem}

In order to compare corresponding results of this paper with Theorem \ref{clas_CKN}, we now state here our CKN type inequalities with a remainder term on $\mathbb{R}^{k}\times \mathbb{R}^{n-k}$: Let $1<p,q<\infty$, $0<r<\infty$ with $p+q\geq r$, $\delta\in [0,1]\cap\left[\frac{r-q}{r},\frac{p}{r}\right]$ and $ \alpha,b,c \in \mathbb{R}$. Assume that 
$ \frac{\delta r}{p}+\frac{(1-\delta)r}{q} = 1, c=-\delta+b(1-\delta)$. Then for any complex-valued function $f \in C^{\infty}_{0}(\mathbb{R}^{n}\backslash\{x'=0\})$, we have   
\begin{multline}\label{Lp_CKN}
    \left|\frac{k-\alpha}{p}\right|^{\delta}\left\||x'|_{k}^{\frac{\alpha c}{p}}f\right\|_{{L^r} {(\mathbb{R}^{n})}}\leq
    \left[\left\|\frac {x' \cdot \nabla_{k}}{|x'|^{\frac{\alpha}{p}}_{k}}f\right\|_{{L^p}(\mathbb {R}^{n})}^{p}\right.\\\left.-\int_{\mathbb{R}^{n}}C_{p}\left( \frac{x'\cdot \nabla_{k}}{|x'|^{\frac{\alpha}{p}}_{k}}f, \frac{x'\cdot \nabla_{k}}{|x'|^{\frac{\alpha}{p}}_{k}}f+\frac{k-\alpha}{p}\frac{f}{|x'|^{\frac{\alpha}{p}}_{k}}\right)dx\right]^{\frac{\delta}{p}}
    \left\||x'|_{k}^{\frac{\alpha b}{p}}f\right\|_{{L^q}(\mathbb {R}^{n})}^{1-\delta}, 
\end{multline}
where the constant $\left|\frac{k-\alpha}{p}\right|^{\delta}$ with $\alpha\neq k$ is sharp when $q=p$ with $b=-1$ or $\delta=\overline{0,1}$.

Here, when $\alpha=p$ and $k=n$ dropping the last term on the right-hand side of \eqref{Lp_CKN} gives the following improved version for all complex-valued functions $f\in C_0^{\infty}(\mathbb{R}^n\backslash\{x'=0\})$ with an explicit constant:
\begin{equation}\label{CKN_improved}
\begin{split}
\||x|_{E}^{c}f\|_{{L^r}(\mathbb{R}^n)} &\leq \left|\frac{p}{n-p}\right|^{\delta}\left\|\frac{x\cdot \nabla f}{|x|_{E}}\right\|^{\delta}_{{L^p}(\mathbb{R}^n)}\left\||x|_{E}^{b}f\right\|_{{L^q}(\mathbb{R}^n)}^{1-\delta} \\& \leq \left|\frac{p}{n-p}\right|^{\delta}\left\|\nabla f\right\|^{\delta}_{{L^p}(\mathbb{R}^n)}\left\||x|_{E}^{b}f\right\|_{{L^q}(\mathbb{R}^n)}^{1-\delta},
\end{split}
\end{equation}
where in the last line we have used the Schwarz inequality. 

In the special case $p=q=r=2$, $b=-n/2$, $c=-\delta-n(1-\delta)/2$ the inequality \eqref{CKN_improved} takes the form
\begin{equation}\label{CKN_special_case}
\left\|\frac{f}{|x|_{E}^{\frac{2\delta+n(1-\delta)}{2}}}\right\|_{L^{2}(\mathbb{R}^n)}
\leq \left|\frac{2}{n-2}\right|^{\delta} \left\|\nabla f\right\|^{\delta}_{L^{2}(\mathbb{R}^n)}
\left\|\frac{f}{|x|_{E}^{\frac{n}{2}}}\right\|^{1-\delta}_{L^{2}(\mathbb{R}^n)}.
\end{equation}
Since we have $1/2+b/n=0$ here, we observe that the condition \eqref{clas_CKN0} is not satisfied, then the inequality \eqref{CKN_special_case} is not covered by the classical Caffarelli-Kohn-Nirenberg inequality, Theorem \ref{clas_CKN}. We also discuss the sharpness of the constant $\left|\frac{k-\alpha}{p}\right|^{\delta}$ in \eqref{Lp_CKN} in some special cases.

Furthermore, in the special case $c=0, r=2, b=1, q=p'=p/(p-1), \delta=1/2, \alpha=p$, dropping the remainder term, we obtain the following cylindrical extension of the Heisenberg-Pauli-Weyl uncertainty principle from \eqref{Lp_CKN}: 
    \begin{equation}\label{C1f}
\left|\frac{k-p}{p}\right| \int_{\mathbb{R}^{n}}|f|^2 d x \leq \left(\int_{\mathbb{R}^{n}}\left|\frac{x'\cdot\nabla_{k}}{|x'|} f\right|^p d x\right)^{\frac{1}{p}}\left(\int_{\mathbb{R}^{n}}|x'|^{p'}|f|^{p'} d x\right)^{\frac{1}{p'}}.
\end{equation}

We also discuss the above results on homogeneous Lie groups for the radial derivative operator with any homogeneous quasi-norm and on stratified Lie groups for a horizontal gradient with the Euclidean norm on the first stratum. These results in the Abelian case $\mathbb{G}=(\mathbb{R}^{n};+)$ imply the (weighted) improved $L^{p}$ Hardy identity \eqref{Lp_Hardy_radial}, the Hardy type identity with logarithmic weight \eqref{Lp_logarithmic} and the CKN inequality \eqref{Lp_CKN} above for the radial derivative operator $df/d|x|$ with respect to any homogeneous quasi-norm $|x|$.

In Section \ref{Preliminaries} we briefly recall all the necessary definitions and notations. Then, (weighted) improved $L^p$-Hardy inequalities and identities for all complex-valued functions $f\in C_{0}^{\infty}(\mathbb{R}^{n}\backslash \{x'=0\})$ on $\mathbb{R}^{k}\times\mathbb{R}^{n-k}$, stratified Lie groups and homogeneous Lie
groups are presented in Section \ref{section3}. In Section \ref{section4} we establish $L^p$-Hardy type inequalities and identities with logarithmic weights on $\mathbb{R}^{k}\times\mathbb{R}^{n-k}$, stratified Lie groups and homogeneous Lie
groups. Finally, applications of the obtained results to extended Caffarelli-Kohn-Nirenberg type and Heisenberg-Pauli-Weyl type uncertainty inequalities are discussed in Section \ref{section5}. 

\section{Preliminaries}\label{Preliminaries} 
In this section we briefly recall the necessary notations and definitions concerning the setting of homogeneous Lie groups following the books \cite{FS-book}, \cite{FR16} and \cite{RS-book}. Also, a few other facts needed for our analysis will be discussed.

\subsection{Stratified Lie groups}
Let us begin with stratified (or a homogeneous Carnot groups) Lie groups. 
\begin{definition}
A Lie group $\mathbb{G}=\left(\mathbb{R}^{n}, \cdot \right)$ is called stratified if it satisfies the conditions:
\begin{itemize}
    \item For some natural numbers $N+N_{2}+\cdots+N_{r}=n$, that is $N=N_{1}$, the decomposition $\mathbb{R}^{n}=\mathbb{R}^{N} \times \cdots \times \mathbb{R}^{N_{r}}$ is valid, and for every positive $\lambda>0$ the dilation $\delta_{\lambda}: \mathbb{R}^{n} \rightarrow \mathbb{R}^{n}$ defined by
$$
\delta_{\lambda}(x)=\delta_{\lambda}\left(x^{\prime}, x^{(2)}, \ldots, x^{(r)}\right):=\left(\lambda x^{\prime}, \lambda^{2} x^{(2)}, \ldots, \lambda^{r} x^{(r)}\right)
$$
is an automorphism of $\mathbb{G}$, where $x^{\prime} \equiv x^{(1)} \in \mathbb{R}^{N}$ and $x^{(k)} \in \mathbb{R}^{N_{k}}$ for $k=2, \ldots, r$.
\item Let $N$ be as defined above and let $X_{1}, \ldots, X_{N}$ be the left invariant vector fields on $\mathbb{G}$ with the property $X_{k}(0)=\left.\frac{\partial}{\partial x_{k}}\right|_{0}$ for $k=1, \ldots, N$. Then 
$$
\operatorname{rank}\left(\operatorname{Lie}\left\{X_{1}, \ldots, X_{N}\right\}\right)=n, \;\; \forall x \in \mathbb{R}^{n}.
$$
\end{itemize} 
Thus, we call the triple $\mathbb{G}=\left(\mathbb{R}^{n}, \cdot, \delta_{\lambda}\right)$ is a stratified Lie group.
\end{definition}

Recall that the left invariant vector fields $X_{1}, \ldots, X_{N}$ are called the (Jacobian) generators of $\mathbb{G}$ and $r$ is called a step of $\mathbb{G}$. The homogeneous dimension of $\mathbb{G}$ is defined by 
$$
Q=\sum_{k=1}^{r} k N_{k}, \quad N_{1}=N.
$$
In this setting, $d x$ denotes the Haar measure on a group $\mathbb{G}$, and it is known that the Haar measure on $\mathbb{G}$ is the standard Lebesgue measure for  $\mathbb{R}^{n}$ (see, e.g. \cite[Proposition 1.6.6]{FR16}).

We also recall that the left invariant vector fields $X_{j}$ have an explicit form and satisfy the divergence theorem, see e.g.  \cite[Section 3.1.5]{FR16} and \cite{RS}, 
\begin{equation}\label{divergence}
X_{k}=\frac{\partial}{\partial x_{k}^{\prime}}+\sum_{\ell=2}^{r} \sum_{m=1}^{N_{1}} a_{k, m}^{(\ell)}\left(x^{\prime}, \ldots, x^{\ell-1}\right) \frac{\partial}{\partial x_{m}^{(\ell)}}.
\end{equation}
The corresponding horizontal gradient and the horizontal divergence are defined by
$$
\nabla_{H}:=\left(X_{1}, \ldots, X_{N}\right)
$$
and
$$
\operatorname{div}_{H} v:=\nabla_{H} \cdot v,
$$
respectively. We will also use the following notations
$$
\left|x^{\prime}\right|=\sqrt{x_{1}^{\prime 2}+\cdots+x_{N}^{\prime 2}}
$$
and 
for the Euclidean norm on $\mathbb{R}^{N}$.

The explicit representation of the left invariant vector fields $X_{j}$ from \eqref{divergence} allows us to establish the identities

\begin{equation}
\left.\left.\left|\nabla_{H}\right| x^{\prime}\right|^{\gamma}|=\gamma| x^{\prime}\right|^{\gamma-1},
\end{equation}
and
\begin{equation}
    \label{cylin_ball}
    B'(0,R):=\{x\in \mathbb{G}:|x'|\leq R\}
\end{equation}
for $R>0$.

\subsection{Homogeneous Lie Groups}
 Now, we recall basic necessary concepts and fix the notations of a general homogeneous Lie group in a very briefly manner.

 A homogeneous Lie group is a connected simply connected Lie group $\mathfrak{g}$ whose Lie algebra is equipped with the following family of dilations:
 $$
D_\lambda=\operatorname{Exp}(A \ln \lambda)=\sum_{k=0}^{\infty} \frac{1}{k !}(\ln (\lambda) A)^k,
$$
where $A$ is a diagonalisable positive linear operator on $\mathfrak{g}$, and each $D_\lambda$ satisfies 
$$
\forall X, Y \in \mathfrak{g}, \;\lambda >0,\; [D_{\lambda}X,D_{\lambda}Y]=D_{\lambda}[X,Y],
$$
where $[X,Y]:=XY-YX$ is the Lie bracket. It induces the dilation structure on $\mathbb{G}$ which we denote by $D_{\lambda x}$ or just by $\lambda x$.

Let $d x$ be the Haar measure on $\mathbb{G}$ and let $|S|$ denote the volume of a measurable set $S \subset \mathbb{G}$. Then we have
$$
\left|D_\lambda(S)\right|=\lambda^Q|S| \text { and } \int_G f(\lambda x) d x=\lambda^{-Q} \int_G f(x) d x,
$$
where $Q:=\operatorname{Tr} A$ is the homogeneous dimension of $\mathbb{G}$.

In this setting we work with a homogeneous quasi-norm on $\mathbb{G}$, which is a continuous non-negative function $\mathbb{G} \ni x \mapsto|x| \in[0, \infty)
$ with the properties 
\begin{itemize}
\item $\left|x^{-1}\right|=|x|$ for all $x \in \mathbb{G}$, 
\item $|\lambda x|=\lambda|x|$ for all $x \in \mathbb{G}$ and $\lambda >0$,
\item $|x|=0$ if and only if $x=0$. 
\end{itemize} 
With respect to the homogeneous quasi-norm $|\cdot|$, we define the radial derivative operator by
\begin{equation}\label{radial_der_def}
\mathcal{R}_{|x|}f:=\frac{df(x)}{d|x|}
\end{equation}
and the quasi-ball centred at $x \in \mathbb{G}$ with radius $R > 0$ by 
$$
B(x, R) := \{y \in \mathbb{G} : |x^{-1}y| < R\}.
$$
The following polar decomposition on homogeneous Lie groups will be useful for our analysis: there is a (unique) positive Borel measure $\sigma$ on the unit sphere
$$
\mathfrak S:=\{x \in \mathbb{G}:|x|=1\}
$$
such that for all $f \in L^1(\mathbb{G})$ we have
$$
\int_{\mathbb{G}} f(x) d x=\int_0^{\infty} \int_{\mathfrak S} f(r y) r^{Q-1} d \sigma(y) d r.
$$

\section{Improved $L^{p}$-Hardy inequalities and identities}\label{section3}
In this section we discuss (weighted) improved $L^{p}$-Hardy inequalities and identities on $\mathbb{R}^{k}\times \mathbb{R}^{n-k}$, stratified and homogeneous Lie groups.

\subsection{Improved $L^{p}$-Hardy inequalities and identities on \texorpdfstring{$\mathbb{R}^{k}\times \mathbb{R}^{n-k}$}{Lg}}

\begin{theorem}\label{Th_Lp_Hardy_cyl}
Let $x=\left(x', x''\right) \in {\mathbb{R}^k} \times \mathbb{R}^{n-k}$. Let $1<p<\infty$, $\alpha \in \mathbb{R}$ and $|\cdot|_{k}$ be the Euclidean norm on $\mathbb{R}^{k}$.
\begin{enumerate}[label=(\roman*)]
\item For all complex-valued functions $f\in C_{0}^{\infty}(\mathbb{R}^{n} \backslash \{x'=0\})$, we have
\begin{equation}\label{Lp_inequality_cyl}
\left|{\frac{{k}-{\alpha}}{p}}\right|^{p}{{\left\|\frac{f}{|x'|_k^{\frac{{\alpha}}{p}}}\right\|^{p}_{L^{p}(\mathbb{R}^{n})} \leq \left\|\frac {x' \cdot \nabla_k}{|x'|_k^{\frac{{\alpha}}{p}}}f\right\|^{p}_{{L}^{p}(\mathbb{R}^{n})}}},
\end{equation}
where the constant $\left|{\frac{{k}-{\alpha}}{p}}\right|^{p}$ is sharp when $\alpha\neq{k}$.
\item
Moreover, we have the identity 
\begin{multline}\label{Lp_identity_cyl}
\left|{\frac{k-\alpha}{p}}\right|^{p}\left\|\frac{f}{|x'|_{k}^{\frac{\alpha}{p}}}\right\|_{L^{p}(\mathbb{R}^{n})}^p=\left\|\frac {x' \cdot \nabla_k}{|x'|_{k}^{\frac{{\alpha}}{p}}}f\right\|_{{L}^{p}(\mathbb{R}^{n})}^p \\-\int_{\mathbb{R}^{n}}C_{p}\left(\frac{x'\cdot\nabla_{k}}{|x'|_{k}^{\frac{\alpha}{p}}}f, \frac{x'\cdot\nabla_{k}}{|x'|_{k}^{\frac{\alpha}{p}}}f+\frac{k-\alpha}{p}\frac{f}{|x'|_{k}^{\frac{\alpha}{p}}}\right)dx
\end{multline}
for all complex-valued functions $f\in C_{0}^{\infty}(\mathbb{R}^{n} \backslash \{x'=0\})$, where the functional $C_{p}(\cdot,\cdot)$ is given by
\begin{equation}\label{Cp}
C_{p}(\xi,\eta):=|\xi|^{p}-|\xi-\eta|^{p}-p|\xi-\eta|^{p-2}\operatorname{Re}(\xi-\eta)\cdot \overline{\eta}\geq 0.
\end{equation}
\end{enumerate}
Furthermore, for all $p\geq2$, the functional $C_p$ vanishes if and only if $$f=|x|_{k}^{-\frac{k-\alpha}{p}}\varphi\left({\frac{x'}{|x'|_{k}}},x''\right)$$ 
for some function $\varphi:\mathbb{S}^{k-1}\times\mathbb{R}^{n-k}\rightarrow\mathbb{C}$, which makes the left-hand side of \eqref{Lp_identity_cyl} infinite unless $\varphi=f=0$ on the basis of the non-integrability of the function
$$\frac{|f(x)|^{p}}{|x'|_{k}^{\alpha}}=\frac{|\varphi(\frac{x'}{|x'|_{k}},x'')|^{p}}{|x'|_{k}^{k}}$$
on $\mathbb{R}^n$. Thus, the sharp constant $\left|{\frac{k-\alpha}{p}}\right|^{p}$ in \eqref{Lp_inequality_cyl} is not attained on $C^\infty_0(\mathbb{R}^n\backslash\{{x'=0}\}).$
\end{theorem}
\begin{remark}
Note that \eqref{Lp_identity_cyl} is a sharp remainder formula for \eqref{Lp_inequality_cyl}.   
\end{remark}
\begin{remark} We refer to \cite{BEHL08} or \cite{OS09} for the special case $\alpha=0$, $k=n$ of the inequality \eqref{Lp_inequality_cyl}. Note that when $k=n$ we have $x'\cdot \nabla_{k}f/|x'|_{k}=df/d|x|$, thus implying corresponding improved Hardy inequalities for the radial derivative operator $df/d|x|$, e.g. the results from e.g. \cite{MOW13} and \cite{MOW19}. As for the cylindrical case, we refer to \cite[Corollary 1.1 and 1.2]{DP21} for $p=2$. 
    \end{remark}
\begin{proof}[Proof of Theorem \ref{Th_Lp_Hardy_cyl}]
By a direct calculation, we observe that
\begin{equation}\label{div_k}
\operatorname{div}_{k}\left(\frac{x^{\prime}}{\left|x^{\prime}\right|^ {\alpha}}\right)=\frac{\sum_{j=1}^{k}\left|x^{\prime}\right|^{\alpha} \partial_{j} x_{j}^{\prime}-\sum_{j=1}^{k} x_{j}^{\prime} \alpha\left|x^{\prime}\right|^{\alpha-1} \partial_{j}\left|x^{\prime}\right|}{\left|x^{\prime}\right|^{2 \alpha}}=\frac{k-\alpha}{\left|x^{\prime}\right|^{ \alpha}}
\end{equation}
for any $\alpha \in \mathbb{R}$ and $\left|x^{\prime}\right| \neq 0$, where $\operatorname{div}_{k}$ is the usual divergence on $\mathbb{R}^{k}$. By using the identity \eqref{div_k} and the divergence theorem one calculates
\begin{equation}\label{used_d_th}
\begin{split}
\int_{\mathbb{R}^{n}}\frac{|f(x)|^p}{|x'|_{k}^{\alpha}} dx &= \frac{1}{k-\alpha} \int_{\mathbb{R}^{n}}|f(x)|^p\operatorname{div}_{k}\left(\frac{x'}{|x'|_{k}^{\alpha}}\right) dx\\&= -\frac{p}{k-\alpha}\operatorname{Re}\int_{\mathbb{R}^{n}}f(x)|f(x)|^{p-2} \frac{\overline{x' \cdot \nabla_k f}}{|x'|_{k}^{\alpha}} dx.
\end{split}
\end{equation}
By the H\"older inequality, it follows that
\begin{equation*}
\begin{split}
\int_{\mathbb{R}^{n}}\frac{|f(x)|^p}{|x'|_{k}^{\alpha}} dx & \leq \left|\frac{p}{k-\alpha}\right|\int_{\mathbb{R}^{n}}\frac{|f(x)|^{p-1}}{|x'|_{k}^{\alpha}}{|x' \cdot \nabla_k f|}dx\\& \leq \left|\frac{p}{k-\alpha}\right|\int_{\mathbb{R}^{n}}\frac{|f(x)|^{p-1}}{|x'|_{k}^{\frac{\alpha(p-1)}{p}}}\frac{|x' \cdot \nabla_k f|}{|x'|_{k}^{\frac{\alpha}{p}}}dx \\& \leq \left|\frac{p}{k-\alpha}\right|\left(\int_{\mathbb{R}^{n}}\frac{|f(x)|^{p}}{|x'|_{k}^{\alpha}}dx\right)^{\frac{p-1}{p}}\left(\int_{\mathbb{R}^{n}}\frac{|x' \cdot \nabla_k f|^p}{|x'|_{k}^{\alpha}}dx\right)^{\frac{1}{p}}
\end{split}
\end{equation*}
which gives inequality \eqref{Lp_inequality_cyl}.

Observe that
$$
h_{1}(x)=|x'|_{k}^{-\frac{k-\alpha}{p}}, \;k\neq \alpha
$$
satisfies the following H\"older’s equality condition
$$
\left|\frac{p}{k-\alpha}\right|^{p}\frac{|x'\cdot \nabla_{k}h_{1}(x)|^{p}}{|x'|_{k}^{\alpha}}=\frac{|h_{1}(x)|^{p}}{|x'|_{k}^{\alpha}}
$$
implying the sharpness of the constant.

Let us show Part (ii).
Applying notations
$$
\xi:=\frac{x'\cdot \nabla_{k}}{|x'|_{k}^{\frac{\alpha}{p}}}f, \quad \eta:=\frac{x'\cdot \nabla_{k}}{|x'|_{k}^{\frac{\alpha}{p}}}f+\frac{k-\alpha}{p}\frac{f}{|x'|_{k}^{\frac{\alpha}{p}}}
$$
formula \eqref{Cp} can be rewritten as
\begin{equation}\label{Cp_1}
\begin{split}
C_{p}(\xi,\eta)=\left|\frac{x'\cdot \nabla_{k}}{|x'|_{k}^{\frac{\alpha}{p}}}f \right|^{p}-\left|-\frac{k-\alpha}{p}\frac{f}{|x'|_{k}^{\frac{\alpha}{p}}} \right|^{p}-p\left|-\frac{k-\alpha}{p}\frac{f}{|x'|_{k}^{\frac{\alpha}{p}}} \right|^{p-2}\operatorname{Re}\left(-\frac{k-\alpha}{p}\frac{f}{|x'|_{k}^{\frac{\alpha}{p}}}\right)\times \\ \overline{\left(\frac{x'\cdot \nabla_{k}f}{|x'|_{k}^{\frac{\alpha}{p}}}+\frac{k-\alpha}{p}\frac{f}{|x'|_{k}^{\frac{\alpha}{p}}}\right)}.
\end{split}
\end{equation}
Let us take integral from the both side of \eqref{Cp_1} and by using the identity \eqref{used_d_th} we have 
\begin{equation}\label{Cp_2}
\begin{split}
\int_{\mathbb{R}^{n}} C_{p}(\xi,\eta)dx=\int_{\mathbb{R}^{n}} \frac{|x'\cdot \nabla_{k}f|^{p}}{|x'|_{k}^{\alpha}}dx-\left|\frac{k-\alpha}{p}\right|^{p}\int_{\mathbb{R}^{n}}\frac{|f(x)|^{p}}{|x'|_{k}^{\alpha}}dx+p\left|\frac{k-\alpha}{p}\right|^{p}\int_{\mathbb{R}^{n}} \frac{|f(x)|^{p}}{|x'|_{k}^{\alpha}}dx
\\ -p\left|\frac{k-\alpha}{p}\right|^{p}\int_{\mathbb{R}^{n}} \frac{|f(x)|^{p}}{|x'|_{k}^{\alpha}}dx=\int_{\mathbb{R}^{n}} \frac{|x'\cdot \nabla_{k}f|^{p}}{|x'|_{k}^{\alpha}}dx-\left|\frac{k-\alpha}{p}\right|^{p}\int_{\mathbb{R}^{n}} \frac{|f(x)|^{p}}{|x'|_{k}^{\alpha}}dx,
\end{split}
\end{equation}
which gives \eqref{Lp_identity_cyl}.

Recalling from \cite[Step 3 of Proof of Lemma 3.4]{CKLL24} that
\begin{equation}\label{Cp_cylind_proof}
C_{p}(\xi,\eta)=|\xi|^{p}-|\xi-\eta|^{p}-p|\xi-\eta|^{p-2}\operatorname{Re}(\xi-\eta)\cdot \overline{\eta}\geq c_{p}|\eta|^{p}, \;\;2\leq p<\infty,
\end{equation}
for some $c_{p}\in(0,1]$, one can observe that $C_{p}(\xi,\eta)=0$ if and only if $\eta=0$. Therefore, the second term on the right-hand side of \eqref{Lp_identity_cyl} vanishes if and only if 
$$
\frac{x'\cdot \nabla_{k}}{|x'|_{k}^{\frac{\alpha}{p}}}f+\frac{k-\alpha}{p}\frac{f}{|x'|_{k}^{\frac{\alpha}{p}}}=0
$$
which is equivalent to
$$
\partial_{|x'|_{k}}\left(|x'|_{k}^{\frac{k-\alpha}{p}}f\right)=0.
$$
Then there exists a function $\varphi$ such that
$$
f(x)=|x|_{k}^{-\frac{k-\alpha}{p}}\varphi\left({\frac{x'}{|x'|_{k}}},x''\right).
$$
In this case 
$$
\frac{|f(x)|^{p}}{|x'|_{k}^{\alpha}}=\frac{|\varphi(\frac{x'}{|x'|_{k}},x'')|^{p}}{|x'|_{k}^{k}},
$$
whose integral
$$
\int_{\mathbb{R}^{n-k}}\int_{\mathbb{R}^{k}}\frac{|\varphi(\frac{x'}{|x'|_{k}},x'')|^{p}}{|x'|_{k}^{k}} dx'dx''
$$
is infinite. So, the left-hand side of \eqref{Lp_identity_cyl} is finite if and only if $\varphi=f=0$, then the sharp constant is not attained.  

The proof of Theorem \ref{Th_Lp_Hardy_cyl} is completed.
\end{proof}

\addcontentsline{toc}{section}{Unnumbered Section}
\subsection{Improved $L^p$-Hardy inequalities and identities on stratified Lie groups} Now, in this subsection we discuss the previous section's results in the setting of stratified Lie groups.
\begin{theorem}\label{Th_Lp_stratified}
Let $\mathbb{G}$ be a stratified Lie group with $\textit{N}$ being the dimension of the first stratum. We denote by $x'$ the variables from the first stratum of $\mathbb{G}$. Let $1<p<\infty$ and $\alpha \in \mathbb{R}$. Then for all complex-valued functions $f\in C_{0}^{\infty}(\mathbb{G}\backslash \{x'=0\})$ we have the following identity 
\begin{multline}\label{Lp_stratified}
\left|{\frac{N-\alpha}{p}}\right|^{p}\left\|\frac{f}{|x'|^{\frac{\alpha}{p}}}\right\|_{L^{p}(\mathbb{G})}^p=\left\|\frac {x' \cdot \nabla_H}{|x'|^{\frac{{\alpha}}{p}}}f\right\|_{{L}^{p}(\mathbb{G})}^p\\-\int_{\mathbb{G}}C_{p}\left(\frac{x'\cdot\nabla_{H}}{|x'|^{\frac{\alpha}{p}}}f, \frac{x'\cdot\nabla_{H}}{|x'|^{\frac{\alpha}{p}}}f+\frac{N-\alpha}{p}\frac{f}{|x'|^{\frac{\alpha}{p}}}\right)dx,
\end{multline}
where the functional $C_{p}(\cdot,\cdot)$ is given by
\begin{equation}\label{Cp'}
C_{p}(\xi,\eta):=|\xi|^{p}-|\xi-\eta|^{p}-p|\xi-\eta|^{p-2}\operatorname{Re}(\xi-\eta)\cdot \overline{\eta}\geq 0.
\end{equation}
\begin{remark}
    If we drop the last term on the right-hand side of \eqref{Lp_stratified}, we obtain the result from \cite{RSY17_strat}.
\end{remark}
\end{theorem}
\begin{proof}[Proof of Theorem \ref{Th_Lp_stratified}]
By a direct calculation, we observe that
\begin{equation}\label{div_H}
\operatorname{div}_{H}\left(\frac{x^{\prime}}{\left|x^{\prime}\right|^ {\alpha}}\right)=\frac{\sum_{j=1}^{N}\left|x^{\prime}\right|^{\alpha} X_{j} x_{j}^{\prime}-\sum_{j=1}^{N} x_{j}^{\prime} \alpha\left|x^{\prime}\right|^{\alpha-1}X_{j}\left|x^{\prime}\right|}{\left|x^{\prime}\right|^{2 \alpha}}=\frac{N-\alpha}{\left|x^{\prime}\right|^{ \alpha}}
\end{equation}
for any $\alpha \in \mathbb{R}$ and $\left|x^{\prime}\right| \neq 0$.

By using the identity \eqref{div_k} and the divergence theorem one calculates
\begin{equation}\label{used_d_th_2}
\begin{split}
\int_{\mathbb{G}}\frac{|f(x)|^p}{|x'|^{\alpha}} dx &= \frac{1}{N-\alpha} \int_{\mathbb{G}}|f(x)|^p\operatorname{div_H}\left(\frac{x'}{|x'|^{\alpha}}\right) dx\\&= -\frac{p}{N-\alpha}\operatorname{Re}\int_{\mathbb{G}}f(x)|f(x)|^{p-2} \frac{\overline{x' \cdot \nabla_H f}}{|x'|^{\alpha}} dx.
\end{split}
\end{equation}
Applying notations
$$
\xi:=\frac{x'\cdot \nabla_{H}}{|x'|^{\frac{\alpha}{p}}}f, \quad \eta:=\frac{x'\cdot \nabla_{H}}{|x'|^{\frac{\alpha}{p}}}f+\frac{N-\alpha}{p}\frac{f}{|x'|^{\frac{\alpha}{p}}}
$$
formula \eqref{Cp'} can be rewritten as
\begin{equation}\label{Cp_1'}
\begin{split}
C_{p}(\xi,\eta)=\left|\frac{x'\cdot \nabla_{H}}{|x'|^{\frac{\alpha}{p}}}f \right|^{p}-\left|-\frac{N-\alpha}{p}\frac{f}{|x'|^{\frac{\alpha}{p}}} \right|^{p}-p\left|-\frac{N-\alpha}{p}\frac{f}{|x'|^{\frac{\alpha}{p}}} \right|^{p-2}\operatorname{Re}\left(-\frac{N-\alpha}{p}\frac{f}{|x'|^{\frac{\alpha}{p}}}\right)\times \\ \overline{\left(\frac{x'\cdot \nabla_{H}f}{|x'|^{\frac{\alpha}{p}}}+\frac{N-\alpha}{p}\frac{f}{|x'|^{\frac{\alpha}{p}}}\right)}.
\end{split}
\end{equation}
Let us take integral from the both side of \eqref{Cp_1'} and by using the identity \eqref{used_d_th_2} we have 
\begin{equation}\label{Cp_2_st}
\begin{split}
\int_{{\mathbb{G}}} C_{p}(\xi,\eta)dx=\int_{{\mathbb{G}}} \frac{|x'\cdot \nabla_{H}f|^{p}}{|x'|^{\alpha}}dx-\left|\frac{N-\alpha}{p}\right|^{p}\int_{{\mathbb{G}}}\frac{|f(x)|^{p}}{|x'|^{\alpha}}dx+p\left|\frac{N-\alpha}{p}\right|^{p}\int_{{\mathbb{G}}} \frac{|f(x)|^{p}}{|x'|^{\alpha}}dx
\\ -p\left|\frac{N-\alpha}{p}\right|^{p}\int_{\mathbb{G}} \frac{|f(x)|^{p}}{|x'|^{\alpha}}dx=\int_{{\mathbb{G}}} \frac{|x'\cdot \nabla_{H}f|^{p}}{|x'|^{\alpha}}dx-\left|\frac{N-\alpha}{p}\right|^{p}\int_{{\mathbb{G}}} \frac{|f(x)|^{p}}{|x'|^{\alpha}}dx,
\end{split}
\end{equation}
which gives \eqref{Lp_stratified}.
\end{proof}

\subsection{Improved $L^p$-Hardy inequalities and identities on homogeneous Lie groups}
Here, we now discuss the above results on a general homogeneous Lie group $\mathbb{G}$ for the radial derivative operator $\mathcal{R}_{|\cdot|}$ with any homogeneous quasi-norm $|\cdot|$ on $\mathbb{G}$. 
\begin{theorem}\label{Th_Lp_hom}
Let $\mathbb{G}$ be a homogeneous Lie group of homogeneous dimension  $\textit{Q}$. Let $1<p<\infty$ and $\alpha \in \mathbb{R}$. Then for all complex-valued functions $f\in C_{0}^{\infty}(\mathbb{G} \backslash \{0\})$, we have
\begin{multline}\label{Lp_hom_id}
\left|{\frac{Q-\alpha}{p}}\right|^{p}\left\|\frac{f}{|x|^{\frac{\alpha}{p}}}\right\|_{L^{p}(\mathbb{G})}^p=\left\|\frac {\mathcal{R}_{|x|}f}{|x|^{\frac{{\alpha}}{p}-1}}\right\|_{{L}^{p}(\mathbb{G})}^p-\int_{\mathbb{G}}C_{p}\left(\frac{\mathcal{R}_{|x|}f}{|x|^{\frac{\alpha}{p}-1}}, \frac{\mathcal{R}_{|x|}f}{|x|^{\frac{\alpha}{p}-1}}+\frac{Q-\alpha}{p}\frac{f}{|x|^{\frac{\alpha}{p}}}\right)dx,
\end{multline}
where the functional $C_{p}(\cdot,\cdot)$ is given by
\begin{equation}\label{Cp''}
C_{p}(\xi,\eta)=|\xi|^{p}-|\xi-\eta|^{p}-p|\xi-\eta|^{p-2}\operatorname{Re}(\xi-\eta)\cdot\overline{\eta}\geq 0.
\end{equation}
Moreover, for all $p\geq2$, the functional $C_p$ vanishes if and only if $$f(x)=|x|^{-\frac{Q-\alpha}{p}}\varphi\left({\frac{x}{|x|}}\right)$$ 
for some function $\varphi:\mathfrak{S}\rightarrow\mathbb{C}$, which makes the left-hand side of \eqref{Lp_hom_id} infinite unless $\varphi=f=0$ on the basis of the non-integrability of the function
$$\frac{|f(x)|^{p}}{|x|^{\alpha}}=\frac{|\varphi(\frac{x}{|x|})|^{p}}{|x|^{Q}}$$
on $\mathbb{G}$. Consequently, the sharp constant $\left|{\frac{Q-\alpha}{p}}\right|^{p}$ in
\begin{equation}\label{Lp_hom_id_ineq}
\left|{\frac{Q-\alpha}{p}}\right|^{p}\left\|\frac{f}{|x|^{\frac{\alpha}{p}}}\right\|_{L^{p}(\mathbb{G})}^p\leq \left\|\frac {\mathcal{R}_{|x|}f}{|x|^{\frac{{\alpha}}{p}-1}}\right\|_{{L}^{p}(\mathbb{G})}^p
\end{equation}
is not attained on $C^\infty_0(\mathbb{G}\backslash\{0\}).$
\end{theorem}
\begin{remark} The sharpness of the constant $\left|{\frac{Q-\alpha}{p}}\right|^{p}$ in \eqref{Lp_hom_id_ineq} was shown in \cite[Theorem 3.4]{RSY_2018} or \cite[Theorem 3.2]{LRY_17} when $\alpha\neq Q$. Moreover, the inequality \eqref{Lp_hom_id_ineq} when $\alpha=0$ and the identity \eqref{Lp_hom_id} when $p=2$ were obtained in \cite[Theorem 3.1 and Theorem 4.1]{RS17_Adv}. We also refer to \cite[Theorem 3.1 and Theorem 4.1]{RSY18} for the inequality \eqref{Lp_hom_id_ineq} with remainder terms.
\end{remark}
\begin{proof}[Proof of Theorem \ref{Th_Lp_hom}] 
By introducing the coordinates $(r,y)=\left(|x|,\frac{x}{|x|}\right)\in (0,\infty)\times \mathfrak S$, we write the integral on the left-hand side of \eqref{Lp_hom_id} as
\begin{equation}\label{Lp_coor}
\begin{split}
\int_{\mathbb{G}}\frac{|f(x)|^p}{|x|^{\alpha}}dx &=\int_{0}^{\infty}\int_{\mathfrak S}r^{Q-1-\alpha}|f(r y)|^{p} d \sigma (y)dr \\& = -\frac{p}{Q-\alpha}\int_{0}^{\infty}r^{Q-\alpha}\operatorname{Re}\int_{\mathfrak S}|f(r y)|^{p-2}f(r y) \overline{\frac{df(ry)}{dr}} d \sigma (y)dr \\&=-\frac{p}{Q-\alpha} \operatorname{Re}\int_{\mathbb{G}}\frac{f(x)|f(x)|^{p-2}}{|x|^{\alpha-1}} \overline{\mathcal{R}_{|x|}f}dx.
\end{split}
\end{equation}
Applying notations 
$$
\xi:=\frac{\mathcal{R}_{|x|}f}{|x|^{\frac{\alpha}{p}-1}},
 \quad 
\eta:=\frac{\mathcal{R}_{|x|}f}{|x|^{\frac{\alpha}{p}-1}}+\frac{Q-\alpha}{p}\frac{f}{|x|^{\frac{\alpha}{p}}}
$$
formula \eqref{Cp''} can be rewritten as
\begin{equation}\label{Cp_homo}
\begin{split}
C_{p}(\xi,\eta)=\left|\frac{\mathcal{R}_{|x|}f}{|x|^{\frac{\alpha}{p}-1}} \right|^{p}-\left|-\frac{Q-\alpha}{p}\frac{\mathcal{R}_{|x|}f}{|x|^{\frac{\alpha}{p}}}\right|^{p}-p\left|-\frac{Q-\alpha}{p}\frac{\mathcal{R}_{|x|}f}{|x|^{\frac{\alpha}{p}}}\right|^{p-2}\operatorname{Re}\left(-\frac{Q-\alpha}{p}\frac{\mathcal{R}_{|x|}f}{|x|^{\frac{\alpha}{p}}}\right)\times \\ \overline{\left(\frac{\mathcal{R}_{|x|}f}{|x|^{\frac{\alpha}{p}-1}}+\frac{Q-\alpha}{p}\frac{\mathcal{R}_{|x|}f}{|x|^{\frac{\alpha}{p}}}\right)}.
\end{split}
\end{equation}
By using the identity \eqref{Lp_coor} we take integral from the both side of \eqref{Cp_homo} to get 
\begin{multline*}
\left|{\frac{Q-\alpha}{p}}\right|^{p}\left\|\frac{f}{|x|^{\frac{\alpha}{p}}}\right\|_{L^{p}(\mathbb{G})}^p=\left\|\frac {\mathcal{R}_{|x|}f}{|x|^{\frac{{\alpha}}{p}-1}}\right\|_{{L}^{p}(\mathbb{G})}^p-\int_{\mathbb{G}}C_{p}\left(\frac{\mathcal{R}_{|x|}f}{|x|^{\frac{\alpha}{p}-1}}, \frac{\mathcal{R}_{|x|}f}{|x|^{\frac{\alpha}{p}-1}}+\frac{Q-\alpha}{p}\frac{f}{|x|^{\frac{\alpha}{p}}}\right)dx,
\end{multline*}
which is the identity \eqref{Lp_hom_id}.

Recall from \eqref{Cp_cylind_proof} that $C_{p}(\xi,\eta)=0$ if and only if $\eta=0$. Therefore, the second term on the right-hand side of \eqref{Lp_hom_id} vanishes if and only if
$$
\frac{f}{|x|^{\frac{\alpha}{p}}}+\frac{p}{Q-\alpha}\cdot\frac{\mathcal{R}_{|x|}f}{|x|^{\frac{\alpha}{p}-1}}=0,
$$
which is equivalent to
$$
\mathcal{R}_{|x|}\left(|x|^{\frac{Q-\alpha}{p}f}\right)=0.
$$
Then there exists a function $\varphi:\mathfrak{S}\rightarrow\mathbb{C}$ such that
$$
f(x)=|x|^{-\frac{Q-\alpha}{p}}\varphi\left({\frac{x}{|x|}}\right).
$$
Then we have
$$
\frac{|f(x)|^{p}}{|x|^{\alpha}}=\frac{|\varphi(\frac{x}{|x|})|^{p}}{|x|^{Q}},
$$
hence one can observe that the left-hand side of \eqref{Lp_hom_id} is finite if and only if $\varphi=f=0$, i.e. the constant is not attained.  
\end{proof}
\section{$L^p$-Hardy type inequalities and identities with logarithmic weights}\label{section4}
In this section we discuss (weighted) $L^{p}$-Hardy type inequalities and identities involving logarithmic weights on $\mathbb{R}^{k}\times \mathbb{R}^{n-k}$, stratified and homogeneous Lie groups. 

\subsection{$L^p$-Hardy type inequalities and identities with logarithmic weights on \texorpdfstring{$\mathbb{R}^{k}\times \mathbb{R}^{n-k}$}{Lg}}
\begin{theorem}\label{th_ball} Let $x=\left(x', x''\right) \in {\mathbb{R}^k} \times \mathbb{R}^{n-k}$. Let $|\cdot|_{k}$ be the Euclidean norm on $\mathbb{R}^{k}$. Let $1<p<\infty$ and $\alpha,\beta \in \mathbb{R}$.  
\begin{enumerate}[label=(\roman*)]
    \item Let $(k-\alpha-p)(\beta+p)\geq 0$. Then for all complex-valued functions $f\in C_0^{\infty}(\{0<|x'|_{k}<R\})$ we have
     \begin{equation}\label{ball_id_ineq}
     \left|\frac{\beta+p}{p}\right|^{p}\quad   \left\|\frac{f}{|x'|_{k}^{\frac{\alpha+p}{p}}\left(log\frac{R}{|x'|_{k}}\right)^{\frac{\beta+p+1}{p}}}\right\|_{L^{p}(0<|x'|_{k}<R)}^p\leq \left\|\frac {x' \cdot \nabla_{k}}{|x'|_{k}^{\frac{\alpha+p}{p}}\left(log\frac{R}{|x'|_{k}}\right)^{\frac{\beta+1}{p}}}f\right\|_{{L}^{p}(0<|x'|_{k}<R)}^p, 
     \end{equation}
     where the constant $\left|\frac{\beta+p}{p}\right|^{p}$ is sharp when $\beta\neq - p$.
     \item Moreover, 
for all complex-valued functions $f\in C_0^{\infty}(\{0<|x'|_{k}<R\})$ we have the identity
 \begin{multline}\label{ball_id}
 \left|\frac{\beta+p}{p}\right|^{p}\quad   \left\|\frac{f}{|x'|_{k}^{\frac{\alpha+p}{p}}\left(log\frac{R}{|x'|_{k}}\right)^{\frac{\beta+p+1}{p}}}\right\|_{L^{p}(0<|x'|_{k}<R)}^p=\left\|\frac {x' \cdot \nabla_{k}}{|x'|_{k}^{\frac{\alpha+p}{p}}\left(log\frac{R}{|x'|_{k}}\right)^{\frac{\beta+1}{p}}}f\right\|_{{L}^{p}(0<|x'|_{k}<R)}^p  \\-(k-\alpha-p)\left(\frac{\beta+p}{p}\right)\left|\frac{\beta+p}{p}\right|^{p-2}\int_{0<|x'|_{k}<R}\frac{|f|^p}{|x'|_{k}^{\alpha+p}\left(log\frac{R}{|x'|_{k}}\right)^{\beta+p}}dx
\\-\int_{0<|x'|_{k}<R}C_p\left(\xi,\eta\right)dx,  
\end{multline}
where  \begin{equation}\label{xi_ball}
\xi:=\frac{x'\cdot\nabla_{k}}{|x'|_{k}^{\frac{\alpha+p}{p}}\left(log\frac{R}{|x'|_{k}}\right)^{\frac{\beta+1}{p}}}f
\end{equation}
and
\begin{equation}\label{eta_ball}
\eta:=\frac{x'\cdot\nabla_{k}}{|x'|_{k}^{\frac{\alpha+p}{p}}\left(log\frac{R}{|x'|_{k}}\right)^{\frac{\beta+1}{p}}}f+\frac{\beta+p}{p}\frac{f}{|x'|_{k}^{\frac{\alpha+p}{p}}\left(log\frac{R}{|x'|_{k}}\right)^{\frac{\beta+p+1}{p}}},
\end{equation} and the functional $C_{p}(\cdot,\cdot)$ is given by
\begin{equation}\label{Cp_ball}
C_{p}(\xi,\eta)=|\xi|^{p}-|\xi-\eta|^{p}-p|\xi-\eta|^{p-2}\operatorname{Re}(\xi-\eta)\cdot\overline{\eta}\geq 0.
\end{equation}
\end{enumerate}
Furthermore, when $k=\alpha+p$ for all $p\geq2$, the functional $C_p$ vanishes if and only if $$f=\left(log\frac{R}{|x'|_{k}}\right)^{\frac{\beta+p}{p}}\varphi\left({\frac{x'}{|x'|_{k}}},x''\right)$$ 
for some function $\varphi:\mathbb{S}^{k-1}\times\mathbb{R}^{n-k}\rightarrow\mathbb{C}$, which makes the left-hand side of \eqref{ball_id_ineq} infinite unless $\varphi=f=0$ on the basis of the non-integrability of the function
$$\frac{|f(x)|^{p}}{|x'|_{k}^{\alpha+p}\left(log\frac{R}{|x'|_{k}}\right)^{\beta+p+1}}=\frac{|\varphi(\frac{x'}{|x'|_{k}},x'')|^{p}}{|x'|_{k}^{k}log\frac{R}{|x'|_{k}}}$$
on $\mathbb{R}^n$. Consequently, the sharp constant $\left|\frac{\beta+p}{p}\right|^{p}$ in \eqref{ball_id_ineq} is not attained on $C_0^{\infty}(\{0<|x'|_{k}<R\})$.
\end{theorem}
\begin{remark} The special case $p=2$, $\beta=-1$ and $\alpha=k-2$ of \eqref{ball_id} was obtained in \cite[Corollary 1.3]{DP21} by the factorization method, which does not work in general $L^{p}$. Taking into account $C_{p}\geq 0$ and dropping the last term in \eqref{ball_id} when $k=n$ yields \cite[Theorem 3]{Sano_arxiv}.
    \end{remark}

\begin{proof}[Proof of Theorem \ref{th_ball}] It is easy to see that when $\beta=-p$ there is nothing to prove. So, let us prove the theorem for $\beta\neq-p$.

Taking into account
\begin{equation*}
   \operatorname{div}_{k}\left(\frac{x'}{|x'|_{k}^{\alpha+p}\left(log\frac{R}{|x'|_{k}}\right)^{\beta+p}}\right)=
   \frac{k-\alpha-p}{|x'|_{k}^{\alpha+p}\left(log\frac{R}{|x'|_{k}}\right)^{\beta+p}}+
   \frac{\beta+p}{|x'|_{k}^{\alpha+p}\left(log\frac{R}{|x'|_{k}}\right)^{\beta+p+1}}
\end{equation*}
we write 
\begin{multline}\label{used_int_by_part}
 \int_{0<|x'|_{k}<R}\frac{|f|^p}{|x'|_{k}^{\alpha+p}\left(log\frac{R}{|x'|_{k}}\right)^{\beta+p+1}} dx \\
 =\frac{1}{\beta+p}\int_{0<|x'|_{k}<R}\left[ \operatorname{div}_{k}\left(\frac{x'}{|x'|_{k}^{\alpha+p}\left(log\frac{R}{|x'|_{k}}\right)^{\beta+p}}\right)-\frac{k-\alpha-p}{{|x'|_{k}^{\alpha+p}\left(log\frac{R}{|x'|_{k}}\right)^{\beta+p}}}\right]|f|^{p} dx \\
 =-\frac{p}{\beta+p}\operatorname{Re}\int_{0<|x'|_{k}<R}\frac{|f|^{p-2}f(\overline{x'\cdot \nabla_{k} f} )}{|x'|_{k}^{\alpha+p}\left(log\frac{R}{|x'|_{k}}\right)^{\beta+p}}dx\\-\frac{k-\alpha-p}{\beta+p}\int_{0<|x'|_{k}<R}\frac{|f|^p}{|x'|_{k}^{\alpha+p}\left(log\frac{R}{|x'|_{k}}\right)^{\beta+p}} dx.
\end{multline}
Now, using H\"older's inequality for $1/p+1/p'=1$ and $(k-\alpha-p)/(\beta+p)\geq 0$ we obtain
\begin{multline}\label{used_int_by_part_Holder}
 \int_{0<|x'|_{k}<R}\frac{|f|^p}{|x'|_{k}^{\alpha+p}\left(log\frac{R}{|x'|_{k}}\right)^{\beta+p+1}} dx \leq
\left|\frac{p}{\beta+p}\right|\int_{0<|x'|_{k}<R}\frac{|f|^{p-1}|x'\cdot \nabla_{k} f|}{|x'|_{k}^{\frac{\alpha+p}{p}+\frac{\alpha+p}{p'}}\left(log\frac{R}{|x'|_{k}}\right)^{\frac{\beta+p+1}{p'}+\frac{\beta+1}{p}}}dx\\
\leq \left|\frac{p}{\beta+p}\right|\left( \int_{0<|x'|_{k}<R}\frac{|f|^p}{|x'|_{k}^{\alpha+p}\left(log\frac{R}{|x'|_{k}}\right)^{\beta+p+1}} dx\right)^{\frac{1}{p'}}
\left( \int_{0<|x'|_{k}<R}\frac{|x'\cdot \nabla_{k} f|^p}{|x'|_{k}^{\alpha+p}\left(log\frac{R}{|x'|_{k}}\right)^{\beta+1}} dx\right)^{\frac{1}{p}}
\end{multline}
yielding \eqref{ball_id_ineq}.

Observe that
$$
h_{2}(x)=\left(log\frac{R}{|x'|_{k}}\right)^{\frac{\beta+p}{p}}, \;k\neq \alpha
$$
satisfies the following H\"older’s equality condition
$$
\left|\frac{p}{\beta+p}\right|^{p}\frac{|x'\cdot \nabla_{k} f|^p}{|x'|_{k}^{\alpha+p}\left(log\frac{R}{|x'|_{k}}\right)^{\beta+1}}=\frac{|h_{2}(x)|^{p}}{|x'|_{k}^{\alpha+p}\left(log\frac{R}{|x'|_{k}}\right)^{\beta+p+1}}
$$
implying the sharpness of the constant.

Let us now prove Part (ii). Recalling the definitions \eqref{xi_ball} and \eqref{eta_ball}, the formula \eqref{Cp_ball} can be rewritten as
\begin{multline}\label{used_not_s}
C_{p}(\xi,\eta)=\left|\frac{x'\cdot\nabla_{k}}{|x'|_{k}^{\frac{\alpha+p}{p}}\left(log\frac{R}{|x'|_{k}}\right)^{\frac{\beta+1}{p}}}f\right|^{p}-\left|-\frac{\beta+p}{p}\frac{f}{|x'|_{k}^{\frac{\alpha+p}{p}}\left(log\frac{R}{|x'|_{k}}\right)^{\frac{\beta+p+1}{p}}}\right|^{p}\\-p\left|-\frac{\beta+p}{p}\frac{f}{|x'|_{k}^{\frac{\alpha+p}{p}}\left(log\frac{R}{|x'|_{k}}\right)^{\frac{\beta+p+1}{p}}}\right|^{p-2}\operatorname{Re}\left(-\frac{\beta+p}{p}\frac{f}{|x'|_{k}^{\frac{\alpha+p}{p}}\left(log\frac{R}{|x'|_{k}}\right)^{\frac{\beta+p+1}{p}}}\right)\\ \times \overline{\left(\frac{x'\cdot\nabla_{k}}{|x'|_{k}^{\frac{\alpha+p}{p}}\left(log\frac{R}{|x'|_{k}}\right)^{\frac{\beta+1}{p}}}f+\frac{\beta+p}{p}\frac{f}{|x'|_{k}^{\frac{\alpha+p}{p}}\left(log\frac{R}{|x'|_{k}}\right)^{\frac{\beta+p+1}{p}}}\right)}.
\end{multline}
Let us take the integral from the both side of \eqref{used_not_s} and by using the identity \eqref{used_int_by_part} we have
\begin{multline*}
\int_{0<|x'|_{k}<R}C_{p}(\xi,\eta)dx=\int_{0<|x'|_{k}<R}\left|\frac{x'\cdot\nabla_{k}}{|x'|_{k}^{\frac{\alpha+p}{p}}\left(log\frac{R}{|x'|_{k}}\right)^{\frac{\beta+1}{p}}}f\right|^{p}dx\\-\left|\frac{\beta+p}{p}\right|^{p}\int_{0<|x'|_{k}<R}\left|\frac{f}{|x'|_{k}^{\frac{\alpha+p}{p}}\left(log\frac{R}{|x'|_{k}}\right)^{\frac{\beta+p+1}{p}}}\right|^{p}dx\\-(k-\alpha-p)\left(\frac{\beta+p}{p}\right)\left|\frac{\beta+p}{p}\right|^{p-2}\int_{0<|x'|_{k}<R}\frac{|f|^p}{|x'|_{k}^{\alpha+p}\left(log\frac{R}{|x'|_{k}}\right)^{\beta+p}} dx,
\end{multline*}
which implies \eqref{ball_id}.

Recall from \eqref{Cp_cylind_proof} that $C_{p}(\xi,\eta)=0$ if and only if $\eta=0$. Therefore, the second term on the right-hand side of \eqref{ball_id} vanishes if and only if 
$$
\frac{x'\cdot\nabla_{k}}{|x'|_{k}^{\frac{\alpha+p}{p}}\left(log\frac{R}{|x'|_{k}}\right)^{\frac{\beta+1}{p}}}f+\frac{\beta+p}{p}\frac{f}{|x'|_{k}^{\frac{\alpha+p}{p}}\left(log\frac{R}{|x'|_{k}}\right)^{\frac{\beta+p+1}{p}}}=0
$$
which is equivalent to
$$
\partial_{|x'|_{k}}\left(\left(log\frac{R}{|x'|_{k}}\right)^{-\frac{\beta+p}{p}}f\right)=0.
$$
Then there exists a function $\varphi$ such that
$$
f(x)=\left(log\frac{R}{|x'|_{k}}\right)^{\frac{\beta+p}{p}}\varphi\left({\frac{x'}{|x'|_{k}}},x''\right).
$$
In this case 
$$\frac{|f(x)|^{p}}{|x'|_{k}^{\alpha+p}\left(log\frac{R}{|x'|_{k}}\right)^{\beta+p+1}}=\frac{|\varphi(\frac{x'}{|x'|_{k}},x'')|^{p}}{|x'|_{k}^{k}log\frac{R}{|x'|_{k}}}$$
whose integral on $\{x=(x',x'')\in \mathbb{R}^{k}\times\mathbb{R}^{n-k}: 0<|x'|_{k}<R\}$ is infinite. So, the left-hand side of \eqref{Lp_identity_cyl} is finite if and only if $\varphi=f=0$, then the sharp constant is not attained. 

The proof of Theorem \ref{th_ball} is completed.
\end{proof}
\subsection{$L^p$-Hardy type inequalities and identities with logarithmic weights on homogeneous Lie groups} In this subsection we establish $L^p$-Hardy type inequalities and identities with logarithmic weights for the radial derivative operator $\mathcal{R}_{|\cdot|}$ with respect to any homogeneous quasi-norm $|\cdot|$ on a general homogeneous Lie group $\mathbb{G}$.
\begin{theorem}\label{log_hom} 
Let $\mathbb{G}$ be a homogeneous Lie group of homogeneous dimension $Q$ and let $|\cdot|$ be any homogeneous quasi-norm on $\mathbb{G}$. Let $1<p<\infty$ and $\alpha,\beta \in \mathbb{R}$. Let $B(0,R) \subset \mathbb{G}$ be a quasi-ball with radius $R>0$. 
\begin{enumerate}[label=(\roman*)]
\item Let $(Q-\alpha-p)(\beta+p)\geq 0$. Then for all complex-valued functions $f\in C_0^{\infty}(B(0,R)\backslash \{0\})$ we have
     \begin{equation}\label{ball_id_ineq_hom}
     \left|\frac{\beta+p}{p}\right|^{p}\quad   \left\|\frac{f}{|x|^{\frac{\alpha+p}{p}}\left(log\frac{R}{|x|}\right)^{\frac{\beta+p+1}{p}}}\right\|_{L^{p}(B'(0,R))}^p\leq \left\|\frac {\mathcal{R}_{|x|}f}{|x|^{\frac{\alpha}{p}}\left(log\frac{R}{|x|}\right)^{\frac{\beta+1}{p}}}f\right\|_{{L}^{p}(B(0,R))}^p, 
     \end{equation}
     where the constant $\left|\frac{\beta+p}{p}\right|^{p}$ is sharp when $\beta\neq - p$.
     \item Moreover, for all complex-valued functions $f\in C_0^{\infty}(B(0,R)\backslash \{0\})$ 
we have the identity 
\begin{multline}\label{log_id_hom}
\left(\frac{\beta+p}{p}\right)^{p}\left\|\frac{f}{|x|^{\frac{\alpha+p}{p}}\left(log\frac{R}{|x|}\right)^{\frac{\beta+p+1}{p}}}\right\|^{p}_{L^{p}(B(0,R))}=
\left\|\frac{\mathcal{R}_{|x|}f}{|x|^{\frac{\alpha}{p}}\left(log\frac{R}{|x|}\right)^{\frac{\beta+1}{p}}}\right\|^{p}_{L^{p}(B(0,R))}\\-(Q-\alpha-p)\left(\frac{\beta+p}{p}\right)\left|\frac{\beta+p}{p}\right|^{p-2}\int_{B(0,R)}\frac{|f|^p}{|x|^{\alpha+p}\left(log\frac{R}{|x|}\right)^{\beta+p}} dx
\\-p \int_{B(0,R)}C_{p}\left(\frac{\mathcal{R}_{|x|}f}{|x|^{\frac{\alpha}{p}}\left(log\frac{R}{|x|}\right)^{\frac{\beta+1}{p}}},\frac{\mathcal{R}_{|x|}f}{|x|^{\frac{\alpha}{p}}\left(log\frac{R}{|x|}\right)^{\frac{\beta+1}{p}}}+\frac{\beta+p}{p}\frac{f}{|x|^{\frac{\alpha+p}{p}}\left(log\frac{R}{|x|}\right)^{\frac{\beta+p+1}{p}}}\right)dx
\end{multline}
for all complex-valued functions $f\in C_0^{\infty}(B(0,R)\backslash \{0\})$, where \begin{equation}\label{Cp_ball_homo}
C_{p}(\xi,\eta)=|\xi|^{p}-|\xi-\eta|^{p}-p|\xi-\eta|^{p-2}\operatorname{Re}(\xi-\eta)\cdot\overline{\eta}\geq 0.
\end{equation}
\end{enumerate}
Furthermore, when $Q=\alpha+p$ for all $p\geq2$, the functional $C_p$ vanishes if and only if $$f=\left(log\frac{R}{|x|}\right)^{\frac{\beta+p}{p}}\varphi\left({\frac{x}{|x|}}\right)$$ 
for some function $\varphi:\mathfrak{S}\rightarrow\mathbb{C}$, which makes the left-hand side of \eqref{ball_id_ineq_hom} infinite unless $\varphi=f=0$ on the basis of the non-integrability of the function
$$\frac{|f(x)|^{p}}{|x|^{\alpha+p}\left(log\frac{R}{|x|}\right)^{\beta+p+1}}=\frac{|\varphi(\frac{x}{|x|})|^{p}}{|x|^{Q}log\frac{R}{|x|}}$$
on $\mathbb{R}^n$. Thus, the sharp constant $\left|\frac{\beta+p}{p}\right|^{p}$ in \eqref{ball_id_ineq_hom} is not attained in $C_0^{\infty}(B(0,R)\backslash \{0\})$.
\end{theorem}
\begin{remark} In the Abelian case $\mathbb{G}=(\mathbb{R}^{n};+)$, hence $Q=n$, dropping the last term in \eqref{log_id_hom} yields the result \cite[Theorem 3]{Sano_arxiv} with any homogeneous quasi-norm instead of the Euclidean norm.
\end{remark}
\begin{proof}[Proof of Theorem \ref{log_hom}]
Taking into account
\begin{equation*}
\mathcal{R}_{r}\left(\frac{r^{Q-\alpha-p}}{\left(log\frac{R}{r}\right)^{\beta+p}}\right)=
\frac{(Q-\alpha-p)r^{Q-\alpha-p-1}}{\left(log\frac{R}{r}\right)^{\beta+p}}+(\beta+p)\frac{r^{Q-\alpha-p-1}}{\left(log\frac{R}{r}\right)^{\beta+p+1}},
\end{equation*}
then introducing polar coordinates $(r, y)=\left(|x|, \frac{x}{|x|}\right) \in$ $(0, \infty) \times \wp$ on $\mathbb{G}$, where the quasi-sphere $\wp$, we have 
\begin{multline*}
\int_{B(0,R)} \frac{|f(x)|^{p}}{|x|^{\alpha+p}\left(log\frac{R}{|x|}\right)^{\beta+p+1}}dx= \int_{0}^{\infty}\int_{\mathfrak S}\frac{|f(ry)|^{p}r^{Q-1-\alpha-p}}{\left(log\frac{R}{r}\right)^{\beta+p+1}}d\sigma(y)dr\\=\frac{1}{\beta+p}\int_{0}^{\infty}\int_{\mathfrak S}\left( \mathcal{R}_{r}\left(\frac{r^{Q-\alpha-p}}{\left(log\frac{R}{r}\right)^{\beta+p}}\right)-\frac{(Q-\alpha-p)r^{Q-\alpha-p-1}}{\left(log\frac{R}{r}\right)^{\beta+p}}\right)|f(ry)|^{p}d\sigma(y)dr. 
\end{multline*}
Then the integrating by parts implies that
\begin{multline*}
\int_{0}^{\infty}\int_{\mathfrak S}\frac{|f(ry)|^{p}r^{Q-1-\alpha-p}}{\left(log\frac{R}{r}\right)^{\beta+p+1}}d\sigma(y)dr \\=\frac{1}{\beta+p}\int_{0}^{\infty}\int_{\mathfrak S}\left( \mathcal{R}_{r}\left(\frac{r^{Q-\alpha-p}}{\left(log\frac{R}{r}\right)^{\beta+p}}\right)-\frac{(Q-\alpha-p)r^{Q-\alpha-p-1}}{\left(log\frac{R}{r}\right)^{\beta+p}}\right)|f(ry)|^{p}d\sigma(y)dr
 \\=-\frac{p}{\beta+p}\operatorname{Re}\int_{0}^{\infty}\int_{\mathfrak S}f(ry)|f(ry)|^{p-2}\overline{{\frac{df(ry)}{dr}}}\frac{r^{Q-\alpha-p}}{\left(log\frac{R}{r}\right)^{\beta+p}}d\sigma(y)dr 
 \end{multline*}
     \begin{multline}\label{used_in_by_parts_hom}
 -\frac{(Q-\alpha-p)}{\beta+p}\int_{0}^{\infty}\int_{\mathfrak S}\frac{r^{Q-\alpha-p-1}}{\left(log\frac{R}{r}\right)^{\beta+p}}|f(ry)|^{p}d\sigma(y)dr\\
 \\
 = -\frac{p}{\beta+p}\operatorname{Re}\int_{B(0,R)}\frac{f(x)|f(x)|^{p-2}\overline{{\mathcal{R}_{|x|}f}}}{
|x|^{p+\alpha-1}\left(log\frac{R}{|x|}\right)^{\beta+p}}dx-\frac{(Q-\alpha-p)}{\beta+p}\int_{B(0,R)}\frac{|f(x)|^{p}}{|x|^{\alpha+p}\left(log\frac{R}{|x|}\right)^{\beta+p}}dx. 
\end{multline}
Now, as in the cylindrical case above, recalling $(Q-\alpha-p)/(\beta+p)\geq 0$ we use the H\"older inequality for $1/p+1/p'=1$ to obtain
\begin{multline}\label{used_int_by_part_Holder_hom}
 \int_{B(0,R)}\frac{|f|^p}{|x|^{\alpha+p}\left(log\frac{R}{|x|}\right)^{\beta+p+1}} dx \leq
\left|\frac{p}{\beta+p}\right|\int_{B(0,R)}\frac{|f|^{p-1}|\mathcal{R}_{|x|} f|}{|x|^{\frac{\alpha}{p}+\frac{\alpha+p}{p'}}\left(log\frac{R}{|x|}\right)^{\frac{\beta+p+1}{p'}+\frac{\beta+1}{p}}}dx\\
\leq \left|\frac{p}{\beta+p}\right|\left( \int_{B(0,R)}\frac{|f|^p}{|x|^{\alpha+p}\left(log\frac{R}{|x|}\right)^{\beta+p+1}} dx\right)^{\frac{1}{p'}}
\left( \int_{B(0,R)}\frac{|\mathcal{R}_{|x|} f|^p}{|x|^{\alpha}\left(log\frac{R}{|x|}\right)^{\beta+1}} dx\right)^{\frac{1}{p}}
\end{multline}
yielding \eqref{ball_id_ineq_hom}. The sharpness of the constant can be shown as in the cylindrical case.

Let us now prove Part (ii). Applying notations
$$
\xi:=\frac{\mathcal{R}_{|x|}f}{|x|^{\frac{\alpha}{p}}\left(log\frac{R}{|x|}\right)^{\frac{\beta+1}{p}}}
$$
and
$$
\eta:=\frac{\mathcal{R}_{|x|}f}{|x|^{\frac{\alpha}{p}}\left(log\frac{R}{|x|}\right)^{\frac{\beta+1}{p}}}+\frac{\beta+p}{p}\frac{f}{|x|^{\frac{\alpha+p}{p}}\left(log\frac{R}{|x|}\right)^{\frac{\beta+p+1}{p}}}
$$
formula \eqref{Cp_ball_homo} can be rewritten as
\begin{multline}\label{used_not_hom}
C_{p}(\xi,\eta)=\left|\frac{\mathcal{R}_{|x|}f}{|x|^{\frac{\alpha}{p}}\left(log\frac{R}{|x|}\right)^{\frac{\beta+1}{p}}}\right|^{p}-\left|-\frac{\beta+p}{p}\frac{f}{|x|^{\frac{\alpha+p}{p}}\left(log\frac{R}{|x|}\right)^{\frac{\beta+p+1}{p}}}\right|^{p}-\\p\left|-\frac{\beta+p}{p}\frac{f}{|x|^{\frac{\alpha+p}{p}}\left(log\frac{R}{|x|}\right)^{\frac{\beta+p+1}{p}}}\right|^{p-2}\operatorname{Re}\left(-\frac{\beta+p}{p}\frac{f}{|x|^{\frac{\alpha+p}{p}}\left(log\frac{R}{|x|}\right)^{\frac{\beta+p+1}{p}}}\right)\\ \times\overline{\left(\frac{\mathcal{R}_{|x|}f}{|x|^{\frac{\alpha}{p}}\left(log\frac{R}{|x|}\right)^{\frac{\beta+1}{p}}}+\frac{\beta+p}{p}\frac{f}{|x|^{\frac{\alpha+p}{p}}\left(log\frac{R}{|x|}\right)^{\frac{\beta+p+1}{p}}}\right)}
\end{multline}
Let us take the integral from the both side of \eqref{used_not_hom} and by using the identity \eqref{used_in_by_parts_hom} we have 
\begin{multline}
\int_{B(0,R)}C_{p}(\xi,\eta)=\int_{B(0,R)}\left|\frac{\mathcal{R}_{|x|}f}{|x|^{\frac{\alpha}{p}}\left(log\frac{R}{|x|}\right)^{\frac{\beta+1}{p}}}\right|^{p}-\int_{B(0,R)}\left|\frac{\beta+p}{p}\frac{f}{|x|^{\frac{\alpha+p}{p}}\left(log\frac{R}{|x|}\right)^{\frac{\beta+p+1}{p}}}\right|^{p}-\\ 
-(Q-\alpha-p)\left(\frac{\beta+p}{p}\right)\left|\frac{\beta+p}{p}\right|^{p-2}\int_{B(0,R)}\frac{|f|^p}{|x|^{\alpha+p}\left(log\frac{R}{|x|}\right)^{\beta+p}} dx
\end{multline}
The equality \eqref{log_id_hom} is proved. 

Recall from \eqref{Cp_cylind_proof} that $C_{p}(\xi,\eta)=0$ if and only if $\eta=0$. Therefore, the second term on the right-hand side of \eqref{log_id_hom} vanishes if and only if 
$$
\frac{\mathcal{R}_{|x|}f}{|x|^{\frac{\alpha}{p}}\left(log\frac{R}{|x|}\right)^{\frac{\beta+1}{p}}}+\frac{\beta+p}{p}\frac{f}{|x|^{\frac{\alpha+p}{p}}\left(log\frac{R}{|x|}\right)^{\frac{\beta+p+1}{p}}}=0
$$
which is equivalent to
$$
\mathcal{R}_{|x|}\left(\left(log\frac{R}{|x|}\right)^{-\frac{\beta+p}{p}}f\right)=0.
$$
Then there exists a function $\varphi:\mathfrak{S}\rightarrow\mathbb{C}$ such that
$$f=\left(log\frac{R}{|x|}\right)^{\frac{\beta+p}{p}}\varphi\left({\frac{x}{|x|}}\right)$$
In this case, we have
$$\frac{|f(x)|^{p}}{|x|^{\alpha+p}\left(log\frac{R}{|x|}\right)^{\beta+p+1}}=\frac{|\varphi(\frac{x}{|x|})|^{p}}{|x|^{Q}log\frac{R}{|x|}}$$
whose integral
$$
\int_{\mathbb{G}}\frac{|f(x)|^{p}}{|x|^{\alpha+p}\left(log\frac{R}{|x|}\right)^{\beta+p+1}} dx
$$
is infinite. So, the left-hand side of \eqref{ball_id_ineq_hom} is finite if and only if $\varphi=f=0$, that is, the sharp constant is not attained. 
\end{proof}

\subsection*{$L^p$-Hardy type inequalities and identities with logarithmic weights on stratified Lie groups}
We extend the $L^p$-Hardy identity with logarithmic type function from the previous subsection to stratified Lie groups.
\begin{theorem}\label{TH_log_st} 
Let $\mathbb{G}$ be a stratified Lie group with $\textit{N}$ being the dimension of the
first stratum. We denote by $x'$ the variables from the first stratum of $\mathbb{G}$. Let $1<p<\infty$ and $\alpha, \beta \in \mathbb{R}$. Let $B'(0, R)$ be as in \eqref{cylin_ball}.
\begin{enumerate}[label=(\roman*)]
    \item Let $(N-\alpha-p)(\beta+p)\geq 0$. Then for all complex-valued functions $f\in C_0^{\infty}(B'(0,R)\backslash \{x'=0\})$ we have
     \begin{equation}\label{ball_id_ineq_strat}
     \left|\frac{\beta+p}{p}\right|^{p}\quad   \left\|\frac{f}{|x'|^{\frac{\alpha+p}{p}}\left(log\frac{R}{|x'|}\right)^{\frac{\beta+p+1}{p}}}\right\|_{L^{p}(B'(0,R))}^p\leq \left\|\frac {x' \cdot \nabla_{H}}{|x'|^{\frac{\alpha+p}{p}}\left(log\frac{R}{|x'|}\right)^{\frac{\beta+1}{p}}}f\right\|_{{L}^{p}(B'(0,R))}^p, 
     \end{equation}
     where the constant $\left|\frac{\beta+p}{p}\right|^{p}$ is sharp when $\beta\neq - p$.
     \item Moreover, for all complex-valued functions $f\in C_0^{\infty}(B'(0,R)\backslash \{x'=0\})$ we have the following identity: 
\begin{multline}\label{log_stratified_id}
\left({\frac{\beta+p}{p}}\right)^{p}\quad \left\|\frac{f}{|x'|^{\frac{\alpha+p}{p}}\left(log\frac{R}{|x'|}\right)^{\frac{\beta+p+1}{p}}}\right\|_{L^{p}(B'(0,R))}^p=\left\|\frac {x' \cdot \nabla_{H}}{|x'|^{\frac{\alpha+p}{p}}\left(log\frac{R}{|x'|}\right)^{\frac{\beta+1}{p}}}f\right\|_{{L}^{p}(B'(0,R))}^p \\-(N-\alpha-p)\left(\frac{\beta+p}{p}\right)\left|\frac{\beta+p}{p}\right|^{p-2}\int_{B'(0,R)}\frac{|f|^p}{|x'|^{\alpha+p}\left(log\frac{R}{|x'|}\right)^{\beta+p}}dx
\end{multline}
\begin{multline*}
-p
\int_{B'(0,R)}C_p\left(\frac{x'\cdot\nabla_{H}f}{|x'|^{\frac{\alpha+p}{p}}\left(log\frac{R}{|x'|}\right)^{\frac{\beta+1}{p}}},\frac{x'\cdot\nabla_{H}f}{|x'|^{\frac{\alpha+p}{p}}\left(log\frac{R}{|x'|}\right)^{\frac{\beta+1}{p}}}
\right.\\\left.+\frac{\beta+p}{p}\frac{f}{|x'|^{\frac{\alpha+p}{p}}\left(log\frac{R}{|x'|}\right)^{\frac{\beta+p+1}{p}}}\right)dx.  
\end{multline*}
\end{enumerate}

\end{theorem}
\begin{remark}
   The inequality \eqref{ball_id_ineq_strat} implies \cite[Theorem 4.1]{RS17_strat} when $p=N$, $\beta=-1$ and $\alpha=0$. Moreover, the identity \eqref{log_stratified_id} gives sharp remainder formula for their result.
\end{remark}
\section{Applications}\label{Section7}\label{section5}
In this section we discuss more general cylindrical Caffarelli-Kohn-Nirenberg (CKN) inequalities with explicit constants and remainder terms, which imply the cylindrical Heisenberg-Pauli-Weyl uncertainty principle.

Let us begin with the CKN inequality on $\mathbb{R}^{k}\times \mathbb{R}^{n-k}$:
\begin{theorem}\label{Theorem_CKN}
   Let $x=(x',x'')\in \mathbb{R}^{k}\times\mathbb{R}^{n-k}$,$1\leq k \leq n$, $1<p,q<\infty$, $0<r<\infty$, with $p+q\geq r$, $\delta\in [0,1]\cap\left[\frac{r-q}{r},\frac{p}{r}\right]$ and $ \alpha,b,c \in \mathbb{R}$. Assume that 
$ \frac{\delta r}{p}+\frac{(1-\delta)r}{q} = 1,\quad c=-\delta+b(1-\delta)$. Let $|\cdot|_{k}$ be the Euclidean norm on $\mathbb{R}^{k}$.
Then for any complex-valued function $f \in C^{\infty}_{0}(\mathbb{R}^{n}\backslash\{x'=0\})$, we have  
\begin{multline}\label{CKN_cyl}
    \left|\frac{k-\alpha}{p}\right|^{\delta}\left\||x'|_{k}^{\frac{\alpha c}{p}}f\right\|_{{L^r} {(\mathbb{R}^{n})}}\leq
    \left[\left\|\frac {x' \cdot \nabla_{k}}{|x'|^{\frac{\alpha}{p}}_{k}}f\right\|_{{L^p}(\mathbb {R}^{n})}^{p}\right.\\\left.-\int_{\mathbb{R}^{n}}C_{p}\left( \frac{x'\cdot \nabla_{k}}{|x'|^{\frac{\alpha}{p}}_{k}}f, \frac{x'\cdot \nabla_{k}}{|x'|^{\frac{\alpha}{p}}_{k}}f+\frac{k-\alpha}{p}\frac{f}{|x'|_{k}^{\frac{\alpha}{p}}}\right)dx\right]^{\frac{\delta}{p}}
    \left\||x'|_{k}^{\frac{\alpha b}{p}}f\right\|_{{L^q}(\mathbb {R}^{n})}^{1-\delta}.  
\end{multline}
Moreover, the constant $\left|\frac{k-\alpha}{p}\right|^{\delta}$ with $\alpha\neq k$ is sharp when $q=p$ with $b=-1$ or $\delta=\overline{0,1}$.
\end{theorem}
\begin{remark}
    As we discussed in \eqref{CKN_improved} above, Theorem \ref{Theorem_CKN} is not covered by the classical CKN inequality in general. It also provides its cylindrical extension with an explicit constant and a remainder term. Moreover, in the special case, it implies cylindrical extension of the classical Heisenberg-Pauli-Weyl uncertainty principle (see \eqref{C1f}).
\end{remark}
\begin{proof}[Proof of Theorem \ref{Theorem_CKN}]
Case $\delta=0$. In this case, we have $q=r$ and $b=c$ by $\frac{\delta r}{p}+{\frac{(1-\delta)r}{q}}=1$ and $c=\delta(a-1)+b(1-\delta)$, respectively. Then, the inequality \eqref{CKN_cyl} reduces to the trivial estimate 
\begin{equation}
    \label{triv_iden}
\left\||x'|_{k}^{\frac{\alpha c}{p}}f\right\|_{{L^r}(\mathbb {R}^{n})}\leq\left\||x'|_{k}^{\frac{\alpha b}{p}}f\right\|_{{L^q}(\mathbb {R}^{n})}.
\end{equation}
Case $\delta=1$. Notice that in this case, we have $p=r$, $c=-1$ and
\begin{multline*}
\left\||x'|_{k}^{\frac{\alpha c}{p}}f\right\|_{{L^r}(\mathbb {R}^{n})}\leq\left|\frac{p}{k-\alpha}\right|\left[\left\|\frac {x' \cdot \nabla_{k}}{|x'|^{\frac{\alpha}{p}}_{k}}f\right\|_{{L^p}(\mathbb {R}^{n})}^{p}\right.\\\left.-\int_{\mathbb{R}^{n}}C_{p}\left( \frac{x'\cdot \nabla_{k}}{|x'|^{\frac{\alpha}{p}}_{k}}f, \frac{x'\cdot \nabla_{k}}{|x'|^{\frac{\alpha}{p}}_{k}}f+\frac{k-\alpha}{p}\frac{f}{|x'|_{k}^{\frac{\alpha}{p}}}\right)dx\right]^{\frac{1}{p}},
\end{multline*}
which follows from Part (ii) of Theorem \ref{Th_Lp_Hardy_cyl}.

Thus, it is sufficient to consider the case $\delta\in(0,1)\cap \left[{\frac{r-q}{r},\frac{p}{r}} \right]$. Taking into account $c=-\delta+b(1-\delta)$, a direct calculation gives
$$
  \left\||x'|_{k}^{\frac{\alpha c}{p}}f\right\|_{{L^r}(\mathbb {R}^{n})}=\left(\int_{\mathbb{R}^{n}} {|x'|_{k}^{\frac{\alpha c r}{p}}}|f(x)|^{r} dx\right)^\frac{1}{r}=
\left(\int_{\mathbb{R}^{n}}\frac{|f(x)|^{\delta r}}{|x'|_{k}^{\frac{\alpha \delta r}{p}}}\cdot\frac{|f(x)|^{\left(1-\delta\right)r}}{|x'|_{k}^{\frac{-\alpha br(1-\delta)}{p}}}dx\right)^{\frac{1}{r}}.
$$ 
Since we have $\delta\in(0,1)\cap{\left[\frac{r-q}{r},\frac{p}{r}\right]}$ and $p+q\geq r$, then by using H\"older's inequality for $\frac{\delta r}{p}+{\frac{(1-\delta)r}{q}}=1$, we obtain 
\begin{equation}\label{hold_equal_CKN}
\begin{split}
\left\||x'|_{k}^{\frac{\alpha c}{p}}f\right\|_{{L^r}(\mathbb{R}^{n})}&\leq\left(\int_{\mathbb{R}^{n}}\frac{|f(x)|^{p}}{|x'|_{k}^{\alpha}} dx\right)^{\frac{\delta}{p}}\left(\int_{\mathbb{R}^{n}}\frac{|f(x)|^{q}}{|x'|_{k}^{\frac{-\alpha bq}{p}}} dx\right)^{\frac{(1-\delta)}{q}}\\&=\left\|\frac{f}{|x'|_{k}^{\frac{\alpha}{p}}}\right\|_{L^p({\mathbb{R}^{n}})}^{\delta}\left\|\frac{f}{{|x'|_{k}}^{-\frac{\alpha b}{p}}}\right\|_{L^q(\mathbb{R}^{n})}^{1-\delta}.
\end{split}
\end{equation}

By \eqref{Lp_identity_cyl} of Theorem \ref{Th_Lp_Hardy_cyl} we obtain 
\begin{multline*}
\left\||x'|_{k}^{\frac{\alpha c}{p}}f\right\|_{L^{r}(\mathbb {R}^{n})}\leq{\left|\frac{p}{k-\alpha}\right|^\delta} \left[\left\|\frac {x' \cdot \nabla_{k}}{|x'|^{\frac{\alpha}{p}}_{k}}f\right\|_{{L^p}(\mathbb {R}^{n})}^{p}\right.\\\left.-\int_{\mathbb{R}^{n}}C_{p}\left( \frac{x'\cdot \nabla_{k}}{|x'|^{\frac{\alpha}{p}}_{k}}f, \frac{x'\cdot \nabla_{k}}{|x'|^{\frac{\alpha}{p}}_{k}}f+\frac{k-\alpha}{p}\frac{f}{|x'|_{k}^{\frac{\alpha}{p}}}\right)dx\right]^{\frac{\delta}{p}}\left\||x'|_{k}^{\frac{\alpha b}{p}}f\right\|_{L^{q}(\mathbb {R}^{n})}^{1-\delta},
\end{multline*}
which is \eqref{CKN_cyl}.

Let us now discuss sharpness of the constant $\left|\frac{p}{k-\alpha}\right|^\delta$ in the inequality \eqref{CKN_cyl}.  
Note that in \eqref{hold_equal_CKN} when $q=p$ and $b=-1$, H\"older's equality condition is satisfied for any function. Therefore, we can conclude that the constant $\left|\frac{p}{k-\alpha}\right|^\delta$ in \eqref{CKN_cyl} is sharp when $q=p$ and $b=-1$. There is nothing to prove in the cases $\delta=0$ and $\delta=1$, since in these cases the inequality \eqref{CKN_cyl} reduces to the identities \eqref{triv_iden} and \eqref{Lp_identity_cyl}, respectively. 
\end{proof}
In the exactly same way as above, but using Theorems \ref{Th_Lp_stratified} and \ref{Th_Lp_hom} instead of Theorem \ref{Th_Lp_Hardy_cyl} one can derive the following Theorems \ref{CKN_stratified} and \ref{CKN_homogeneous}, respectively.

\begin{theorem}\label{CKN_stratified}
Let $\mathbb{G}$ be a stratified Lie group with $\textit{N}$ being the dimension of the first stratum. We denote by $x'$ the variables from the first stratum of $\mathbb{G}$. Let $1<p,q<\infty$, $0<r<\infty$ with $p+q\geq r$ and $\delta\in[0,1]\cap\left[\frac{r-q}{r},\frac{p}{r}\right]$ and $\alpha,b,c\in\mathbb{R}$. Assume that ${\frac{\delta r}{p}}+{\frac{(1-\delta)r}{q}}=1$ and $c=-\delta+b(1-\delta)$. 
Then for all complex-valued functions $f\in C_0^{\infty}(\mathbb{G}\backslash \{0\})$ we have  
\begin{multline}\label{CKN_2}
    \left|\frac{N-\alpha}{p}\right|^{\delta}\left\||x'|^{\frac{\alpha c}{p}}f\right\|_{{L^r} {(\mathbb{G})}}\leq
    \left[\left\|\frac {x' \cdot \nabla_{H}}{|x'|^{\frac{\alpha}{p}}}f\right\|_{{L^p}(\mathbb {G})}^{p}\right.\\\left.-\int_{\mathbb{G}}C_{p}\left( \frac{x\cdot \nabla_{H}}{|x|^{\frac{\alpha}{p}}}f, \frac{x\cdot \nabla_{H}}{|x|^{\frac{\alpha}{p}}}f+\frac{N-\alpha}{p}\frac{f}{|x|^{\frac{\alpha}{p}}}\right)dx\right]^{\frac{\delta}{p}}
    \left\||x'|^{\frac{\alpha b}{p}}f\right\|_{{L^q}(\mathbb {G})}^{1-\delta}.  
\end{multline}
Moreover, the constant $\left|\frac{N-\alpha}{p}\right|^{\delta}$ with $\alpha\neq N$ is sharp when $q=p$ with $b=-1$ or $\delta=\overline{0,1}$.
\end{theorem}
\begin{remark}
     Taking into account $C_p\geq 0$ by \eqref{Cp_} and dropping the last term in \eqref{CKN_2}, we obtain \cite[Theorem 4.1]{RSY17_strat}. 
\end{remark}

\begin{theorem}\label{CKN_homogeneous}
Let $\mathbb{G}$ be a homogeneous Lie group of homogeneous dimension $Q$
and let $|\cdot|$ be any homogeneous quasi-norm on $\mathbb{G}$. Let $1<p,q<\infty$,$0<r<\infty$ with $p+q\geq r$ and $\delta\in[0,1]\cap\left[\frac{r-q}{r},\frac{p}{r}\right]$ and $\alpha,b,c\in\mathbb{R}$. Assume that ${\frac{\delta r}{p}}+{\frac{(1-\delta)r}{q}}=1$ and $c=-\delta+b(1-\delta)$. 
Then for all complex-valued functions $f\in C_0^{\infty}(\mathbb{G}\backslash \{0\})$ we have  
\begin{multline}\label{CKN_homo}
\left|\frac{Q-\alpha}{p}\right|^{\delta}\left\||x|^{\frac{\alpha c}{p}}f\right\|_{{L^r} {(\mathbb{G})}}\leq
\left[\left\|\frac {\mathcal{R}_{|x|}f}{|x|^{\frac{\alpha}{p}-1}}\right\|_{{L^p}(\mathbb {G})}^{p}\right.\\\left.-\int_{\mathbb{G}}C_{p}\left( \frac{\mathcal{R}_{|x|}f}{|x|^{\frac{\alpha}{p}-1}}f, \frac{\mathcal{R}_{|x|}f}{|x|^{\frac{\alpha}{p}-1}}f+\frac{Q-\alpha}{p}\frac{f}{|x|^{\frac{\alpha}{p}}}\right)dx\right]^{\frac{\delta}{p}}\left\||x|^{\frac{\alpha b}{p}}f\right\|_{{L^q}(\mathbb {G})}^{1-\delta}.  
\end{multline}
Moreover, the constant $\left|\frac{Q-\alpha}{p}\right|^{\delta}$ with $\alpha\neq N$ is sharp for any homogeneous quasi-norm $|\cdot|$ when $q=p$ with $b=-1$ or $\delta=\overline{0,1}$.
\end{theorem}
\begin{remark}
     Taking into account $C_p\geq 0$ by \eqref{Cp_} and dropping the last term in \eqref{CKN_homo}, we obtain \cite[Formula (7.1) of Theorem 7.1]{RSY18} and \cite[Formula (6) of Theorem 2.1]{RSY17}. 
\end{remark}
\bibliographystyle{plain}
\bibliography{main}

\end{document}